\documentclass{article}

\usepackage[margin=1in]{geometry}
\usepackage[all]{xy}
\usepackage{amsmath,amssymb,amsthm}
\usepackage{chngcntr}
\usepackage{color}
\usepackage{comment}
\usepackage{enumerate}
\usepackage{graphicx}
\usepackage{mathrsfs}
\usepackage{mathtools}
\usepackage{paralist}
\usepackage{amscd}
\usepackage[OT2,T1]{fontenc}
\DeclareSymbolFont{cyrletters}{OT2}{wncyr}{m}{n}
\DeclareMathSymbol{\Sha}{\mathalpha}{cyrletters}{"58}

\theoremstyle{definition}

\newtheorem{theorem}{Theorem}
\newtheorem*{theorem*}{Theorem}
\newtheorem{proposition}[theorem]{Proposition}
\newtheorem*{proposition*}{Proposition}
\newtheorem{lemma}[theorem]{Lemma}
\newtheorem*{lemma*}{Lemma}
\newtheorem{corollary}[theorem]{Corollary}
\newtheorem*{corollary*}{Corollary}
\newtheorem{example}[theorem]{Example}
\newtheorem*{example*}{Example}
\newtheorem{definition}[theorem]{Definition}
\newtheorem*{definition*}{Definition}
\newtheorem{remark}[theorem]{Remark}
\newtheorem*{remark*}{Remark}

\newtheorem*{fact*}{Fact}

\newtheorem*{observation*}{Obzervation}

\newtheorem*{claim*}{Claim}

\newtheorem*{notation*}{Notation}

\newtheorem*{convention*}{Convention}

\newtheorem*{assumption*}{Assumption}

\numberwithin{theorem}{section} 	 
\numberwithin{equation}{section}	 

\newcommand{\Z}{\mathbb{Z}}
\newcommand{\Q}{\mathbb{Q}}
\newcommand{\R}{\mathbb{R}}
\newcommand{\C}{\mathbb{C}}
\newcommand{\F}{\mathbb{F}}
\newcommand{\Sl}{\mathop{\mathrm{Sl}}\nolimits}
\newcommand{\Cl}{\mathop{\mathrm{Cl}}\nolimits}
\newcommand{\la}{\lambda}
\newcommand{\La}{\Lambda}
\newcommand{\al}{\alpha_{\lambda}}
\newcommand{\pe}{\varpi_{1}}
\newcommand{\E}{\mathscr{E}}
\newcommand{\zt}[1]{\zeta \left(#1 \right)}
\newcommand{\sZ}[1]{\textsf{Z}\left(#1 \right)}
\newcommand{\fr}[2]{\frac{#1}{#2}}
\newcommand{\ro}{\rho}
\newcommand{\rob}{\overline{\rho}}
\newcommand{\ol}[1]{\overline{#1}}
\newcommand{\x}[1]{\wp\left(#1\right)}
\newcommand{\y}[1]{\wp'\left(#1\right)}
\newcommand{\s}[1]{\Sl\left(#1\right)}
\newcommand{\vp}[1]{\varphi\left(#1\right)}

\newcommand{\sO}{\mathscr{O}}
\newcommand{\ASl}{\mathrm{ArcSl}}
\newcommand{\p}{\mathfrak{P}}

\begin{document}

\title{Congruence Relations Connecting Tate-Shafarevich Groups\\   with Bernoulli-Hurwitz Numbers by Elliptic Gauss Sums\\ in Eisenstein Integers Case}

\author{Akihiro GOTO}
\date{}
\maketitle

\begin{abstract}

There are classical congruences between the class number of the imaginary quadratic field $\Q(\sqrt{-p})$ for a rational prime $p>3$ and a Bernoulli number or an Euler number. Under the BSD conjecture on the $2$-parts of the leading term, Onishi obtained an elliptic generalization of these congruences, which gives congruences between the order of the Tate-Shafarevich group of certain elliptic curves with complex multiplication by the Gaussian integers $\Z[\sqrt{-1}]$ and Mordell-Weil rank 0, and a coefficient of power series expansion of an elliptic function associated to $\Z[\sqrt{-1}]$. In this paper, we provide Onishi's type congruences for the Eisenstein integers $\Z[\fr{-1+\sqrt{-3}}{2}]$.
\end{abstract}

\tableofcontents


\addcontentsline{toc}{section}{Introduction}
\section*{Introduction}
There is a classical theorem as follows:
\begin{theorem}(\cite{H}) \label{ct}
Let $p>3$ be a rational prime, $h(-p)$ the class number of $\Q(\sqrt{-p})$, then
$$
h(-p)\equiv \left\{
\begin{array}{ll}
-2B_{(p+1)/2} \bmod p & \text{if \ } p\equiv 3\bmod 4, \\
2^{-1}E_{(p-1)/2} \bmod p & \text{if \ }p\equiv 1\bmod 4,
\end{array}
\right.
$$
where $B_{n}$ and $E_{n}$ are the $n$-th Bernoulli and Euler numbers respectively, i.e., $\fr{t}{e^t-1}=\sum_{n=0}^{\infty}B_{n}\fr{t^n}{n!}$, $\fr{2}{e^t+e^{-t}}=\sum_{n=0}^{\infty}E_{n}\fr{t^n}{n!}$.
\end{theorem}

As an elliptic generalization of these congruences, based on Asai's previous research of elliptic Gauss sums \cite{Asai}, Onishi studied a certain family $\{E_{\la}\}$ of elliptic curves over $\Q(i)$, with complex multiplication by $\Z[i]$ and Mordell-Weil rank $0$, parametrized by degree one primes $\la\in \Z[i]$, and showed\footnote{By using \cite{R}, the main result of \cite{O} can be easily improved in this form.} \cite{O} that the congruence relation
$$
\#\Sha(E_{\la}/\Q(i))\underset{\bmod \ell}{\equiv} \text{Hurwitz-type number}
$$
up to multiplication by powers of $2$, where $\ell$ is the norm prime of $\la$ satisfying $\ell\equiv 5 \bmod 8$, holds by proving the congruences between the coefficients of elliptic Gauss sums and Hurwitz-type numbers (defined as power series coefficients of certain elliptic functions), and that, moreover, under the BSD conjecture on the $2$-part of the leading term for the elliptic curve $E_{\la}$, the congruence above for $E_{\la}$ holds without ambiguity of multiplication by powers of $2$. Here, $\Sha(E_{\la}/\Q(i))$ is the Tate-Shafarevich group of the elliptic curve $E_{\la}/\Q(i)$.
The Hurwitz-type number of an elliptic function (with the period lattice $\Z[i]$) is an analogue of the Bernoulli number or the Euler number.


In this paper, we prove an analogue of Onishi's result for the Eisenstein integer ring $\Z[\fr{-1+\sqrt{-3}}{2}]$. To explain the main results of this paper precisely, we first introduce some notation and definitions: let $\ro:=e^{\fr{2\pi i}{3}}$, $\pe:=\displaystyle \int_{0}^{1}\fr{dt}{\sqrt[3]{(1-t^3)^2}}$ (the real period of the elliptic curve $y^2=4x^3-27$; see Lemma~\ref{ee} \eqref{weq}), and $\vp{u}:=\frac{6\wp(\pe u)}{9+\wp'(\pe u)}$, $\Sl(u):=\vp{\fr{u}{-3\pe}}$, where $\wp$ is the Weierstrass $\wp$ function with the period lattice $\pe\Z[\ro]$. We define the rational numbers $C_{n}, D_{n}$ as 
$$
\Sl(u)=:\sum_{n=0}^{\infty}C_{3n+1}u^{3n+1},\ \ \Sl(u)^{-1}=:\frac{1}{u}+\sum_{n=0}^{\infty}D_{3n+2}u^{3n+2}.
$$
($C_{3n+1}, D_{3n+2}$ have easy recurrence relations, respectively; see the proof of Lemma~\ref{denc} and Remark~\ref{remD}.)
The rational numbers $C_{3n+1}, D_{3n+2}$ of elliptic functions (with the period lattice $\Z[\ro]$) are analogues of the Bernoulli numbers or the Euler numbers and we call them {\bf Bernoulli-Hurwitz-type numbers}, named after Katz' paper \cite{Katz}. Let $\ell$ be a rational prime number such that $\ell\equiv 1\bmod 3$, $\la \in \Z[\ro]$ a prime element such that $\ell=\la\ol{\la}$ and $\la\equiv 1\bmod 3$.
We denote by $\chi_{\la}$ the cubic residue character over $(\Z[\rho]/(\la))^{\times}\simeq \F_{\ell}^{\times}$ and set $\widetilde{\la}:=\prod_{\chi_{\la}(\nu)=1}\vp{\fr{\nu}{\la}}$. (Then $\widetilde{\la}^3=\la$ holds; see Lemma~\ref{Epro} \eqref{E2}.) Let 
$$
G_{\la}(\chi_{\la}, \varphi):=\fr{1}{3}\sum_{\nu\in (\Z[\ro]/(\la))^{\times}}\chi_{\la}(\nu)\varphi\left(\fr{\nu}{\la}\right),\ \ G_{\la}(\chi_{\la}, \varphi^{-1}):=\fr{1}{3}\sum_{\nu \in (\Z[\ro]/(\la))^{\times}}\chi_{\la}(\nu)\varphi\left(\fr{\nu}{\la}\right)^{-1}.
$$
$G_{\la}(\chi_{\la}, \varphi)\ (\text{resp.}\ G_{\la}(\chi_{\la}, \varphi^{-1}))$ is called the {\bf elliptic Gauss sum} associated to $\chi_{\la}$ and $\varphi$ (resp. $\varphi^{-1}$). We define 
$$
\al:=\left\{
\begin{array}{ll}
G_{\la}(\chi_{\la}, \varphi)/\widetilde{\la}^2 & \ell \equiv 7 \bmod 9, \\
G_{\la}(\chi_{\la}, \varphi^{-1})/\widetilde{\la}^2 & \ell \equiv 4 \bmod 9.
\end{array}
\right.
$$
The number $\al$ is called the {\bf coefficient of the elliptic Gauss sum}. (Then, $\al \in \chi_{\la}(3)(1+3\Z)$ or $\al \in \ol{\chi_{\la}}(3)(-1+3\Z)$ hold; see Theorem~\ref{defco}.)
The following two theorems are the main results of this paper: 
\begin{theorem*} {\rm (Theorem~\ref{main1})}
The denominator of $C_{\fr{2}{3}(\ell-1)}$ (resp. $ D_{\fr{2}{3}(\ell-1)}$) cannot be divided by $\ell$ when $\ell\equiv 7 \bmod 9$ (resp. $\ell\equiv 4 \bmod 9$). We have
$$
\al \underset{\bmod \la}{\equiv} \left\{
\begin{array}{ll}
-\frac{1}{3}C_{\frac{2}{3}(\ell-1)} & \ell \equiv 7 \bmod 9, \\
-\frac{1}{3}D_{\frac{2}{3}(\ell-1)} & \ell \equiv 4 \bmod 9.
\end{array}
\right.
$$
\end{theorem*}

\begin{theorem*} {\rm (Theorem~\ref{bsd} and Corollary~\ref{mainco})} \label{main2}
We write $\la=3m+1+3n\rho,\ m,n\in \Z$.
Then, for the elliptic curve $\E_{\la}:y^2=x^3+\fr{\la^2}{4}$, we have 
\begin{equation} 
\#\Sha(\E_{\la}/\Q(\ro))
=\left\{
\begin{array}{lll} \label{0B2}
|\al|^2/4& \underset{\bmod \ell}{\equiv}\left(3^{\fr{\ell-4}{3}}D_{\fr{2(\ell-1)}{3}}\right)^2/4 & \text{$\ell \equiv 4 \bmod 9$ and $n \equiv 0\bmod 3$}, \\
|\al|^2& \underset{\bmod \ell}{\equiv}\left(3^{\fr{\ell-4}{3}}D_{\fr{2(\ell-1)}{3}}\right)^2 & \text{$\ell \equiv 4 \bmod 9$ and $n \equiv \pm1 \bmod 3$}, \\
|\al|^2& \underset{\bmod \ell}{\equiv}\left(3^{\fr{\ell-7}{6}}C_{\fr{2(\ell-1)}{3}}\right)^2 &  \text{$\ell \equiv 7 \bmod 9$,}
\end{array}
\right.
\end{equation}
up to multiplication by powers of $2$, $3$. Moreover, for $p=2,3$, under the BSD conjecture on the $p$-part of the leading term for the elliptic curve $\E_{\la}$, the congruence \eqref{0B2} holds without the ambiguity of multiplication by powers of $p$.
\end{theorem*}
Note that $\E_{\la}$ is an elliptic curve over $\Q(\ro)$ with complex multiplication by $\Z[\ro]$ and Mordell-Weil rank $0$, by \cite{CW} and the non-vanishing of $L(\E_{\la}/\Q(\ro),1)$ (due to \cite{Asai}), where $L(\E_{\la}/\Q(\ro),s)$ is the $L$-function of the elliptic curve $\E_{\la}/\Q(\ro)$ (see the proof of Theorem~\ref{bsd}).

We explain the outline of the proof of the main results roughly. Corollary~\ref{mainco} is proved by verifying the following relationship:
$$
\xymatrix{
\text{Bernoulli-Hurwitz-type number} \ar[r]_-{(1)} & \text{Elliptic Gauss sum} \ar[l] \ar[r]_-{(2)} & \text{Hecke $L$-value} \ar[r]_-{(3)} \ar[l] & \text{Elliptic $L$-value} \ar[d]_-{(4)} \ar[l] \\
 & & & \#\Sha \ar[u]
}
$$
The relationship (1) is the content of Theorem~\ref{main1}, which is proved by taking $\vp{\fr{1}{\la}}$-adic expansions of $\vp{\fr{\nu}{\la}}^{\pm1}$ and the elliptic Gauss sums and taking $\bmod \ \vp{\fr{1}{\la}}^{\fr{2(\ell-1)}{3}+1}$.
The relationship (2) is due to Asai \cite{Asai}.
The relationship (3) is the equality $L(s, \widetilde{\chi_{\la}})L(s, \ol{\widetilde{\chi_{\la}}})=L(\E_{\la}/\Q(\ro), s)$ given in Proposition~\ref{Deu}, which is directly proved by the cubic reciprocity law and expressing the local $L$-factors of $L(\E_{\la}/\Q(\ro), s)$ in terms of Jacobi sums, where $\widetilde{\chi_{\la}}$ is the Hecke character associated to $\chi_{\la}$, and $L(s, \widetilde{\chi_{\la}})$ is the Hecke $L$-series associated to $\widetilde{\chi_{\la}}$.
The relationship (4) is the content of Theorem~\ref{bsd}, which is shown by using Rubin's result \cite{R} and calculations of the real period of $\E_{\la}$, local Tamagawa numbers, and $\#\E_{\la}(\Q(\ro))_{\text{tors}}$.

The proofs of these results are similar to the ones of the main results of \cite{O}, approximately.
But there are some differences of the proofs, the phenomena and the outputs between this paper and \cite{O}.
First, the denominator of $d_{6n+2}$, which appears in Theorem~\ref{main1}, is easy, i.e., a power of $3$ (see Lemma~\ref{dend}), although the denominator of $H_{4n}$ in \cite{O} is more arithmetical (see \cite{Katz}).
Next, Theorem~\ref{bsd} implies $|\al|^2$ is divisible by $4$ for $\ell \equiv 4 \bmod 9$ and $n \equiv 0 \bmod 3$, while Proposition 3.9 of \cite{O} in the Gaussian integers case does not imply such divisibility. 

Note that the results of \cite{O} and this paper are studied only over the Gaussian and Eisenstein integers. Since the fields $\Q(i)$ and $\Q(\ro)$ are imaginary quadratic fields and cyclotomic fields at once, we think that the situation is easy. To extend the results to the imaginary quadratic field $\Q(\sqrt{-p})$ for 
a rational prime $p>3$, we may have to consider the inclusion $\Q(\sqrt{-p})\subset \Q(\zeta_{p})$ if $p\equiv 3 \bmod 4$, $\Q(\sqrt{-p}) \subset \Q(\zeta_{4p})$ if $p\equiv 1 \bmod 4$ at some point. 
\ \\

Finally, we explain the organization of this paper.
In Chapter 1, we study properties of certain elliptic functions whose period lattice is the Eisenstein integers. We use these properties to define the elliptic Gauss sums and their coefficients. In Chapter 2, we introduce the elliptic Gauss sums and show the main congruences (Theorem~\ref{main1}) between the Bernoulli-Hurwitz-type numbers and the coefficients of the elliptic Gauss sums.
In Chapter 3, we study certain Hecke characters, the Hecke $L$-series associated to them, certain elliptic curves $\E_{\la}$ for $\la$, and their $L$-functions relevant to our elliptic Gauss sums. Moreover, we show the main results (Theorem~\ref{bsd} and Corollary~\ref{mainco}).
In Appendix, we show some lemmas,  formulae and tables of $C_{n}, D_{n}$, etc. which are not used in the proofs of the main results of this paper but are of independent interest.

\section*{Acknowledgments}

The author would like to express his great gratitude to his advisors Professors Akio Tamagawa and Go Yamashita for their careful and constant advice for the study of this paper. Especially, Professor Yamashita suggested the research of elliptic Gauss sums to the author and corrected many errors of the draft. Moreover, the author is very thankful to them for guiding the author in the master seminar.

\section{Elliptic functions} \label{Ef}
In Chapter \ref{Ef}, we study properties of certain elliptic functions whose period lattice is the Eisenstein integers (or its constant multiple). In \S 1.1, we introduce elliptic functions $\Sl$, $\Cl$, $\varphi$, and $\psi$, and their function-theoretic properties. In \S 1.2, we study power series expansions of $\Sl(ru), \Sl(ru)^{-1} (r \in \Z\setminus \{0\})$ in terms of $\Sl(u)$, which will be used to show the main congruences in Chapter \ref{EGs}. In \S 1.3, we introduce Bernoulli-Hurwitz-type numbers, which will also be used to show the main congruences in Chapter \ref{EGs}. In \S 1.4, we study some arithmetic properties of the values of $\Sl(u)$ at division points.

Let $\ro$ be the cubic root $e^{2\pi i/3}=\fr{-1+\sqrt{-3}}{2}$ of unity, where the imaginary part of $\sqrt{-3}$ is $>0$.
\subsection{The basic properties of elliptic functions in the case of Eisenstein integers}
We denote by $\pe$ (See \cite[\S 1.1.1.]{Asai}) the real period given by
$$
\pe=\displaystyle \int_{0}^{1}\fr{dt}{\sqrt[3]{(1-t^3)^2}}=1.76663875 \cdots ,
$$
and let $\wp(u)$ denote Weierstrass' $\wp$ function with the period lattice $\pe\Z[\ro]$, so that $\wp'^{2}=4\wp^3-27$ (See Lemma~\ref{ee}). In this paper, $\overline{(\cdot)}$ denotes the complex conjugate. We write $$
W:=\{\pm 1, \pm \ro, \pm \rob\},\ W':=\{1,\ro,\rob\}.
$$

\begin{lemma} (The real period and the beta function)\label{pe}
$\pe=\fr{1}{3}B\left(\fr{1}{3},\fr{1}{3}\right)=\fr{1}{2\pi\sqrt{3}}\Gamma \left(\fr{1}{3}\right)^3=\fr{2^{\fr{1}{3}}}{3\sqrt{3}}B(\fr{1}{6},\fr{1}{2})$, where $B$ and $\Gamma$ denote the beta and gamma functions respectively.
\end{lemma}
\begin{proof}
Let $g$ denote $\Gamma(\fr{1}{3})$. Then we have $\Gamma(\fr{2}{3})=\fr{2\pi}{\sqrt{3}g}$ by the formula $\Gamma(s)\Gamma(1-s)=\fr{\pi}{\mathrm{sin}(\pi s)}$.
We also have $\Gamma(\fr{1}{6})=\fr{\sqrt{3}}{\sqrt{\pi}2^{\fr{1}{3}}}g^2$, since $\Gamma(\fr{1}{6})=\fr{2\sqrt{\pi}}{2^{\fr{1}{3}}}\fr{\Gamma(\fr{1}{3})}{\Gamma(\fr{2}{3})}=\fr{2\sqrt{\pi}}{2^{\fr{1}{3}}}\fr{g}{2\pi/\sqrt{3}g}=\fr{\sqrt{3}}{\sqrt{\pi}2^{\fr{1}{3}}}g^2$ by putting $s=\fr{1}{6}$ in the formula $\Gamma(s)\Gamma(s+\fr{1}{2})=\fr{2\sqrt{\pi}}{2^{2s}}\Gamma(2s)$.
Then, $B(\fr{1}{6},\fr{1}{2})=\fr{\Gamma(\fr{1}{6})\Gamma(\fr{1}{2})}{\Gamma(\fr{2}{3})}=\fr{(\sqrt{3}/\sqrt{\pi}2^{\fr{1}{3}})g^2\sqrt{\pi}}{2\pi/\sqrt{3}g}=\fr{3}{2^{\fr{4}{3}}\pi}g^3$
, by the formula $\Gamma(\fr{1}{2})=\sqrt{\pi}$.
By definition and the formula $B(s,t)=\fr{\Gamma(s)\Gamma(t)}{\Gamma(s+t)}$,
\begin{align*}
\pe&=\int_{0}^{1}\fr{dt}{\sqrt[3]{(1-t^3)^2}}=\fr{1}{3}\int_{0}^1x^{-\fr{2}{3}}(1-x)^{-\fr{2}{3}}dx=\fr{1}{3}B\left(\fr{1}{3},\fr{1}{3}\right)=\fr{1}{3}\fr{g^2}{\Gamma(\fr{2}{3})}=\fr{g^3}{2\pi\sqrt{3}}=\fr{2^{\fr{1}{3}}}{3\sqrt{3}}B\left(\fr{1}{6},\fr{1}{2}\right),
\end{align*}
where $x=t^3$.
\end{proof}

\begin{lemma} (Elementary properties of $\wp$-function with the period lattice $\pe \Z[\ro]$) \label{ee}  We have
\begin{enumerate}[(1)]
\item \label{rwp} $\x{\ro u}=\ro \x{u}, \x{-u}=\x{u}, \y{\ro u}=\y{u}, \y{-u}=-\y{u}$,
\item \label{cj} $\wp(\overline{u})=\overline{\wp(u)}, \wp'(\overline{u})=\overline{\wp'(u)}$, thus if $u\in \mathbb{R}$, $\wp(u), \wp'(u)\in \mathbb{R}$,

\item The function $\wp$ satisfies $\wp'^2=4\wp^3-27$, \label{weq}
\item $\x{\fr{\pe}{3}}=3,\ \x{\fr{\ro}{3}\pe}=3\ro,\ \x{\fr{\rob}{3}\pe}=3\rob,\ \x{\fr{1-\ro}{3}\pe}=0,$\label{val3}
\item $\y{\fr{\pe}{3}}=\y{\fr{\ro}{3}\pe}=\y{\fr{\rob}{3}\pe}=-9,\ \y{\fr{1-\ro}{3}\pe}=-3\sqrt{-3}$. \label{val3'}

\end{enumerate}
\end{lemma}
\begin{proof}
\eqref{rwp}
Since $\x{u}=\fr{1}{u^2}+\sum_{0\neq \omega \in \pe \Z[\ro]}\left(\fr{1}{(u-\omega)^2}-\fr{1}{\omega^2}\right)$ and $\ro, -1 \in (\Z[\ro])^{\times}$,
\begin{align*}
\x{\ro u}&=\fr{1}{(\ro u)^2}+\sum_{0\neq \omega \in \pe \Z[\ro]}\left(\fr{1}{(\ro u-\ro \omega)^2}-\fr{1}{(\ro \omega)^2}\right)=\fr{1}{\ro^2}\left(\fr{1}{u^2}+\sum_{0\neq \omega \in \pe \Z[\ro]}\left(\fr{1}{(u-\omega)^2}-\fr{1}{\omega^2}\right)\right)=\fr{1}{\ro^2}\x{u}=\ro \x{u},\\
\x{-u}&=\fr{1}{(-u)^2}+\sum_{0\neq \omega \in \pe \Z[\ro]}\left(\fr{1}{(-u-(-\omega))^2}-\fr{1}{(-\omega)^2}\right)=\fr{1}{u^2}+\sum_{0\neq \omega \in \pe \Z[\ro]}\left(\fr{1}{(u-\omega)^2}-\fr{1}{\omega^2}\right)=\x{u}.
\end{align*}
Differentiating both hand sides of
$
\x{\ro u}=\ro \x{u}, \x{-u}=\x{u},
$
we also get
$
\y{\ro u}=\y{u}, \y{-u}=-\y{u}
$ respectively.

\eqref{cj}
It is obvious, since $\overline{\Z[\ro]}=\Z[\ro]$.

\eqref{weq}
$\wp'^2=4\wp^3-60G_{4}\wp-140G_{6}$, where
$G_{6}:=\sum_{0\neq \omega \in \pe \Z[\ro]}\fr{1}{\omega^6}$, and
$$
G_{4}:=\sum_{0\neq \omega \in \pe \Z[\ro]}\fr{1}{\omega^4}
=\sum_{0\neq \omega \in \pe \Z[\ro]}\fr{1}{(\ro \omega)^4}
=\fr{1}{\ro^4}\sum_{0\neq \omega \in \pe \Z[\ro]}\fr{1}{\omega^4}
=\overline{\ro}\sum_{0\neq \omega \in \pe \Z[\ro]}\fr{1}{\omega^4}=\overline{\ro}G_{4}.
$$
Thus, $G_{4}=0$ and $\wp'^2=4\wp^3-140G_{6}=4(\wp-\wp(\fr{1}{2}\pe))(\wp-\wp(\fr{\ro}{2}\pe))(\wp-\wp(\fr{1+\ro}{2}\pe))$. We denote $e:=\wp(\fr{\pe}{2})$.\\
 Then, since $\ro e=\wp(\fr{\ro}{2}\pe),\  \overline{\ro}e=\wp(\fr{\overline{\ro}}{2}\pe)=\wp(-\fr{1+\ro}{2}\pe)=\wp(\fr{1+\ro}{2}\pe)$ by \eqref{rwp}, we have $4e^3=4\wp(\fr{\pe}{2})\wp(\fr{\ro}{2}\pe)\wp(\fr{1+\ro}{2}\pe)=140G_{6}$.
We also have, by Lemma~\ref{pe},
\begin{align*}
\fr{\pe}{2}&=\int_{0}^{\fr{\pe}{2}}du=\int_{\infty}^{e}\fr{d(\wp(u))}{\wp'(u)}=\int_{e}^{\infty}\fr{dx}{\sqrt{4x^3-4e^3}} =\fr{1}{2e^{\fr{1}{2}}}\int_{1}^{\infty}\fr{dx}{\sqrt{x^3-1}}\\
&=\fr{1}{2e^{\fr{1}{2}}}\fr{1}{3}\int_{0}^{1}y^{-\fr{5}{6}}(1-y)^{-\fr{1}{2}}dy=\fr{1}{2e^{\fr{1}{2}}}\fr{1}{3}B\left(\fr{1}{6},\fr{1}{2}\right)=\fr{1}{e^{\fr{1}{2}}}\fr{\sqrt{3}}{2^{\fr{4}{3}}}\pe,
\end{align*}
where $y=x^{-3}$. Hence $4e^3=27$.

\eqref{val3}
By \eqref{rwp} and the periodicity of $\wp$, we have $\wp(\fr{1-\ro}{3}\pe)=\wp(\fr{1+2\ro}{3}\pe)=\wp(\fr{(1+\ro)+\ro}{3}\pe)$
$=\wp(\fr{-\ro^2+\ro}{3}\pe)=\ro\wp(\fr{-\ro+1}{3}\pe)$.
Hence,
\begin{equation} \label{z}
\wp\left(\fr{1-\ro}{3}\pe\right)=0.
\end{equation}
By $\wp'^2=4\wp^3-27$, we have
\begin{equation} \label{27}
\wp'\left(\fr{1-\ro}{3}\pe\right)^2=-27.
\end{equation}
In particular, we have
\begin{equation} \label{ir}
\wp'\left(\fr{1-\ro}{3}\pe\right)\in i\mathbb{R}.
\end{equation}
By putting $u_{1}=\fr{1-\ro}{3}\pe$ and $u_{2}=\fr{\ro}{3}\pe$ in the addition formula (see, for example, \cite[Chap. XX, 20.3]{WW})
$$
\wp(u_{1}+u_{2})=-\wp(u_{1})-\wp(u_{2})+\fr{1}{4}\left(\fr{\wp'(u_{1})-\wp'(u_{2})}{\wp(u_{1})-\wp(u_{2})}\right)^2,
$$
\eqref{rwp} and \eqref{z}, we have

\begin{align*}
\wp\left(\fr{\pe}{3}\right)
&=-\wp\left(\fr{1-\ro}{3}\pe\right)-\wp\left(\fr{\ro}{3}\pe\right)+\fr{1}{4}\left(\fr{\wp'\left(\fr{1-\ro}{3}\pe\right)-\wp'\left(\fr{\ro}{3}\pe\right)}{\wp\left(\fr{1-\ro}{3}\pe\right)-\wp\left(\fr{\ro}{3}\pe \right)}\right)^2 \\
&=-\ro \wp\left(\fr{\pe}{3}\right)+\fr{1}{4}\left(\fr{\wp'\left(\fr{1-\ro}{3}\pe\right)-\wp'\left(\fr{\pe}{3}\right)}{-\ro\wp\left(\fr{\pe}{3}\right)}\right)^2,
\end{align*}
thus, by \eqref{27},
\begin{align}
-4\ro \wp\left(\fr{\pe}{3}\right)^3&=(1+\ro)\wp\left(\fr{\pe}{3}\right)4\left(-\ro\wp\left(\fr{\pe}{3}\right)\right)^2=\left(\wp'\left(\fr{1-\ro}{3}\pe\right)-\wp'\left(\fr{\pe}{3}\right)\right)^2 \nonumber \\
&=-27+\wp'\left(\fr{\pe}{3}\right)^2-2\wp'\left(\fr{1-\ro}{3}\pe\right)\wp'\left(\fr{\pe}{3}\right). \label{val3''}
\end{align}
By taking the real part of the equation \eqref{val3''}, we get
$$
2\wp\left(\fr{\pe}{3}\right)^3=-27+\wp'\left(\fr{\pe}{3}\right)^2=-27+4\wp\left(\fr{\pe}{3}\right)^3-27,
$$
since \eqref{weq} and \eqref{ir} and $\wp(\fr{\pe}{3}),\wp'(\fr{\pe}{3})\in \mathbb{R}$ by \eqref{cj}.
Hence we have $\wp(\fr{\pe}{3})^3=27$. Since $\wp(\fr{\pe}{3})\in \mathbb{R}$, we obtain $\wp(\fr{\pe}{3})=3$.
From $\wp(\fr{\pe}{3})=3$ and \eqref{rwp}, we get also $\wp \left(\fr{\ro}{3}\pe\right)=3\ro$ and $\wp \left(\fr{\overline{\ro}}{3}\pe\right)=3\overline{\ro}$.


\eqref{val3'}
The divisor of $\wp'(u)$ is $-3(0)+(\fr{\pe}{2})+(\fr{\ro \pe}{2})+(\fr{(1+\ro)\pe}{2})\ (\bmod \pe \Z[\ro])$.
Then, the function $(0, \pe)\to \mathbb{R},\ x\mapsto \wp'(x)$ has a unique zero at $x=\fr{\pe}{2}$, where $(0, \pe)$ is an open interval. Since $\lim_{x\to +0}\wp'(x)=-\infty$ from $\wp'(u)=-\fr{1}{u^3}+\cdots$ around $u=0$, it holds $\wp'(x)<0$ for $x\in (0, \fr{\pe}{2})$. Especially $\wp'(\fr{\pe}{3})<0$.
Since $\wp(\fr{\pe}{3})=3$ (by \eqref{val3}) and $\wp'^2=4\wp^3-27$ (by \eqref{weq}), we get $\wp'(\fr{\pe}{3})^2=81$. Then, we get $\wp'(\fr{\pe}{3})=-9$.
By \eqref{rwp}, we also get $\wp'(\fr{\pe}{3})=\wp'(\fr{\ro}{3}\pe)=\wp'(\fr{\overline{\ro}}{3}\pe)=-9$.
By \eqref{val3''} and \eqref{val3}, we get
$$
-4\ro\cdot 27=-4\ro\wp\left(\fr{\pe}{3}\right)^3=-27+\wp'\left(\fr{\pe}{3}\right)^2-2\wp'\left(\fr{1-\ro}{3}\pe\right)\wp'\left(\fr{\pe}{3}\right)=-27+(-9)^2-2\wp'\left(\fr{1-\ro}{3}\pe\right)(-9).
$$
Hence we have $\wp'(\fr{1-\ro}{3}\pe)=\fr{-4\ro\cdot27+27-(-9)^2}{18}=-3(2\ro+1)=-3\sqrt{-3}$.

\end{proof}

\begin{definition} (\cite[Difinition1.1, 1.2]{Asai} for $\textsf{Z}, \varphi, \psi$)
Let $\zt{u}$ denote Weierstrass' $\zeta$ function associated to the lattice $\pe \Z[\ro]$. We define
\begin{align*}
\sZ{u}&:=\zt{\pe u}-\fr{2\pi}{\sqrt{3}\pe}\ol{u},\\
\varphi(u)&:=\fr{1}{3}\left\{\sZ{u-\fr{1}{3}}+\rob\sZ{u-\fr{\ro}{3}}+\ro\sZ{u-\fr{\rob}{3}}\right\}, \\
\psi(u)&:=-\fr{1}{3}\left\{\sZ{u-\fr{1}{3}}+ \ro\sZ{u-\fr{\ro}{3}}+\rob\sZ{u-\fr{\rob}{3}}\right\},\\
\Sl(u)&:=\varphi\left(\fr{u}{-3\pe}\right),\ \ \Cl(u):=\psi\left(\fr{u}{-3\pe}\right).
\end{align*}
\end{definition}
Note also that \begin{equation} \label{addz}
\sZ{u+v}=\sZ{u}+\sZ{v}+\fr{1}{2}\fr{\wp'(\pe u)-\wp'(\pe v)}{\wp(\pe u)-\wp(\pe v)},
\end{equation}
which follows from the addition formula (see, for example, \cite[Chap. XX, 20.53, example 2]{WW})
\begin{equation} \label{addzt}
\zt{u+v}=\zt{u}+\zt{v}+\fr{1}{2}\fr{\wp'(u)-\wp'(v)}{\wp(u)-\wp(v)}.
\end{equation}

We recall the basic properties of these functions stated in \cite{Asai} without proofs:
\begin{lemma} (Elementary properties of $\zeta, \textsf{Z}, \varphi, \psi, \Sl$, and $\Cl$)(cf. \cite[\S 1.1.1]{Asai} for \eqref{phi} and \eqref{r}--\eqref{o}) \label{pro}
The following hold:
\begin{enumerate}[(1)]
\item $\zt{\ro u}=\rob \zt{u},\ \zt{-u}=-\zt{u},\ \zt{\fr{\pe}{2}}=\fr{\pi}{\sqrt{3}\pe},\ \zt{\fr{\ro}{2}\pe}=\fr{\pi}{\sqrt{3}\pe}\rob$. \label{zt}

\item The function $\textsf{Z}$\ is a non-holomorphic periodic function with the period lattice $ \Z[\ro]$, and satisfies $\sZ{\ro u}=\rob\sZ{u},\ \sZ{-u}=-\sZ{u},\ \sZ{\fr{1}{3}}=1,\ \sZ{\fr{1-\ro}{3}}=0$. \label{Z}
\item $\displaystyle \varphi(u)=\fr{6\wp(\pe u)}{9+\wp'(\pe u)}, \psi(u)=\fr{-9+\wp'(\pe u)}{9+\wp'(\pe u)}$. Thus, $\varphi, \psi$ are elliptic functions with the period lattice $\Z[\ro]$, and $\Sl(u), \Cl(u)$ are elliptic functions with the period lattice $3\pe\Z[\ro]$.
\label{phi}

\item $\displaystyle \wp(\pe u)=\fr{3\varphi(u)}{1-\psi(u)},\wp'(\pe u)=9\fr{1+\psi(u)}{1-\psi(u)}$. \label{psi}


\item $\varphi(\ro u)=\ro \varphi(u), \Sl(\ro u)=\ro \Sl(u), \psi(\ro u)=\psi(u), \Cl(\ro u)=\Cl(u)$. \\
$\varphi(\ol{u})=\ol{\varphi(u)}, \Sl(\ol{u})=\ol{\Sl(u)}, \psi(\ol{u})=\ol{\psi(u)}, \Cl(\ol{u})=\ol{\Cl(u)}$, thus, $\varphi(u), \Sl(u), \psi(u), \Cl(u)\in \R$ for $u\in \R$.\label{r}


\item $\varphi'(u)=-3\pe \psi(u)^2, \psi'(u)=3\pe \varphi(u)^2, \Sl'(u)=\Cl(u)^2, \Cl'(u)=-\Sl(u)^2$. \label{d}


\item $\varphi(u)^3+\psi(u)^3=1$, $\Sl(u)^3+\Cl(u)^3=1, \vp{0}=0, \psi(0)=1, \Sl(0)=0, \Cl(0)=1$. \label{h}


\item $\displaystyle \vp{-u}=-\fr{\vp{u}}{\psi(u)},\ \Sl(-u)=-\fr{\Sl(u)}{\Cl(u)}$, $\psi(-u)=\fr{1}{\psi(u)}$, $\Cl(-u)=\fr{1}{\Cl(u)}$. \label{-}

 
\item $\Sl(-3\pe u)^{-1}+\Sl(3\pe u)^{-1}=\vp{u}^{-1}+\vp{-u}^{-1}=3\x{\pe u}^{-1}$.
\label{g}


\item $\displaystyle \zt{(1-\ro)\pe u}=(1-\rob)\zt{\pe u}+(1-\rob)(\vp{u}^{-1}-\vp{-u}^{-1})$. \label{o}

\end{enumerate}
\end{lemma}
\begin{proof}
\eqref{zt}
The claims $\zt{\ro u}=\rob\zt{u},\ \zt{-u}=-\zt{u}$ are proved similarly as Lemma~\ref{ee} \eqref{rwp}. Thus it suffices to show that $\zt{\fr{\pe}{2}}=\fr{\pi}{\sqrt{3}\pe}.$
This is proved by Legendre's relation (see, for example, \cite[Chap.XX, 20.411]{WW} )
$$
\fr{\ro \pe}{2}\zt{\fr{\pe}{2}}-\fr{\pe}{2}\zt{\fr{\ro \pe}{2}}=\fr{\pi i} {2}.
$$

\eqref{Z}
The calims $\sZ{\ro u}=\rob\sZ{u},\ \sZ{-u}=-\sZ{u}$ are trivial by the definition and \eqref{zt}. By the quasi-periodicity (see, for example, \cite[Chap.XX 20.41]{WW})
$$
\zt{\pe(u+1)}=\zt{\pe u}+2\zt{\fr{\pe}{2}},
$$
we have
$$
\sZ{u+1}-\sZ{u}=\zt{\pe (u+1)}-\zt{\pe u}-\fr{2\pi}{\sqrt{3}\pe}=2\zt{\fr{\pe}{2}}-\fr{2\pi}{\sqrt{3}\pe}=0.
$$
Then, we also have $\sZ{u+\ro}-\sZ{u}=\sZ{\ro (\rob u+1)}-\sZ{\ro \cdot \rob u}=\rob (\sZ{\rob u+1}-\sZ{\rob u})=0$.
Thus, it is proved that $\textsf{Z}$\ is periodic with the period lattice $ \Z[\ro]$. The claim $\sZ{\fr{1-\ro}{3}}=0$ is similarly proved as $\x{\fr{1-\ro}{3}\pe}$ in Lemma~\ref{ee} \eqref{val3}. The claim $\sZ{\fr{1}{3}}=1$ is proved by putting $u=\fr{1-\ro}{3}, v=\fr{\ro}{3}$ in \eqref{addz}, $\sZ{\fr{1-\ro}{3}}=0, \sZ{\ro u}=\rob \sZ{u}$, and Lemma~\ref{ee} \eqref{val3} and \eqref{val3'}.

\eqref{phi}
For $w \in \{1,\ro,\rob\}=W'$, by putting $v=-\fr{w}{3}$ in \eqref{addz}, 
$$
\sZ{u-\fr{w}{3}}=\sZ{u}+\sZ{-\fr{w}{3}}+\fr{1}{2}\fr{\wp'(\pe u)-\wp'(-\fr{w}{3}\pe)}{\wp(\pe u)-\wp(-\fr{w}{3}\pe)}=\sZ{u}-\ol{w}+\fr{1}{2}\fr{\wp'(\pe u)-9}{\wp(\pe u)-3w},
$$
by \eqref{Z}, and Lemma~\ref{ee} \eqref{val3} and \eqref{val3'}. Then, by the definition of $\varphi$, the facts that $\sum_{w\in W'}w=\sum_{w\in W'}\ol{w}=0$, $\sum_{w\in W'}\fr{\ol{w}}{x-w}=\fr{3x}{x^3-1}$, and Lemma~\ref{ee} \eqref{weq}, we have
\begin{align*}
3\varphi(u)&=\sum_{w\in W'}\ol{w}\sZ{u-\fr{w}{3}}=\sum_{w\in W'}\ol{w}\sZ{u}-\sum_{w\in W'}w+\fr{\wp'(\pe u)-9}{2}\sum_{w\in W'}\fr{\ol{w}}{\wp(\pe u)-3w}\\
&=\fr{\wp'(\pe u)-9}{2}\sum_{w\in W'}\fr{\ol{w}}{\wp(\pe u)-3w}=\fr{\wp'(\pe u)-9}{2}\fr{9\x{\pe u}}{\x{\pe u}^3-27}=\fr{18\x{\pe u}}{\y{\pe u}+9}.
\end{align*}
Thus $\varphi(u)=\fr{6\x{\pe u}}{\y{\pe u}+9}$. Similarly, by $\sum_{w\in W'}\fr{w}{x-w}=\fr{3}{x^3-1}$, we have
\begin{align*}
-3\psi(u)&=\sum_{w\in W'}w\sZ{u-\fr{w}{3}}=\sum_{w\in W'}w\sZ{u}-\sum_{w\in W'}1+\fr{\wp'(\pe u)-9}{2}\sum_{w\in W'}\fr{w}{\wp(\pe u)-3w}\\
&=-3+\fr{\wp'(\pe u)-9}{2}\sum_{w\in W'}\fr{w}{\wp(\pe u)-3w}=-3+\fr{\wp'(\pe u)-9}{2}\fr{27}{\x{\pe u}^3-27}=-3+\fr{54}{\y{\pe u}+9}.
\end{align*}
Thus $\psi(u)=1-\fr{18}{\y{\pe u}+9}=\fr{\y{\pe u}-9}{\y{\pe u}+9}$.

\eqref{psi}
This is proved by solving two equations of \eqref{phi} in terms of $\x{\pe u}$ and $\y{\pe u}$.

\eqref{r}
It is trivial, since Lemma~\ref{ee} \eqref{rwp}, \eqref{cj}, and \eqref{phi}.

\eqref{d}
This is proved by differentiating two equations of  \eqref{phi} and the facts that $\wp'^2=4\wp^3-27$ and $\wp''=6\wp^2$.

\eqref{h}
By \eqref{d}, we have
$$
(\varphi^3+\psi^3)'=3(\varphi^2\varphi'+\psi^2\psi')=9\pe(-\varphi^2\psi^2+\psi^2\varphi^2)=0.
$$
Thus, $\varphi^3+\psi^3$ is constant. On the other hand, $\varphi(0)=0, \psi(0)=1$, since \eqref{phi}, and the fact $\x{u},\y{u}$ have poles of order $2,3$ at $u=0$ respectively. Therfore, we have $\varphi^3+\psi^3=1$. By definitions of $\Sl$ and $\Cl$, we also have  $\Sl^3+\Cl^3=1$ and $\Sl(0)=0, \Cl(0)=1$.

\eqref{-}
By \eqref{phi} and Lemma~\ref{ee} \eqref{rwp}, it follows that 
$$
\vp{-u}=\fr{6\x{-\pe u}}{9+\y{-\pe u}}=\fr{6\x{\pe u}}{9-\y{\pe u}}=\fr{9+\y{\pe u}}{9-\y{\pe u}}\fr{6\x{\pe u}}{9+\y{\pe u}}=-\fr{\vp{u}}{\psi(u)}.
$$
By definitions of $\Sl$ and $\Cl$, we also have $\Sl(-u)=-\fr{\Sl(u)}{\Cl(u)}$. By \eqref{phi}, it follows that $\psi(-u)\psi(u)=\fr{-9+\y{\pe u}}{9+\y{\pe u}}\cdot \fr{-9-\y{\pe u}}{9-\y{\pe u}}=1$. By definition of $\Cl$, we also have $\Cl(-u)=\fr{1}{\Cl(u)}$.

\eqref{g}
This is obvious from \eqref{phi}.

\eqref{o}
By substituting $u:=\pe u, v:=-\ro \pe u$ in \eqref{addzt}, \eqref{zt}, and Lemma~\ref{ee} \eqref{rwp}, we have
\begin{align*}
\zt{(1-\ro)\pe u}&=
(1-\rob)\zt{\pe u}+\fr{1}{1-\ro}\fr{\y{\pe u}}{\x{\pe u}}.
\end{align*}
On the other hand, by \eqref{phi} and Lemma~\ref{ee} \eqref{rwp}, we also have
\begin{align*}
\vp{u}^{-1}-\vp{-u}^{-1}&=
\fr{1}{3}\fr{\y{\pe u}}{\x{\pe u}}.
\end{align*}
Thus, \eqref{o} follows from these two equalities.

\end{proof}

\subsection{Power series expansions in terms of $\Sl$}

Let us consider power series expansions in terms of the function $\Sl(u)$ in this section, which will be used to show the main congruences (Theorem~\ref{main1}).

We note that $\Q[[\Sl(u)]]$ is isomorphic to a formal power series ring over $\Q$. Indeed, since $\Sl(u)$ is an elliptic function which is holomorphic at $u=0$, $\Sl(u)$ has infinitely many values around $u=0$. Thus, $\Sl(u)\in \Q[[u]]$ is a transcendental element. First, we will show the following:

\begin{lemma}[Addition formula of $\Sl(u)$] (cf. \cite[\S 1. Appendix. 1. Addition Formula (i)]{Asai}) \label{addsl'}

We have
$$
\Sl(u+v)=\displaystyle \fr{\Sl(u)^2\Cl(v)-\Sl(v)^2\Cl(u)}{\Sl(u)\Cl(v)^2-\Sl(v)\Cl(u)^2}\in \Sl(u)+\Sl(v)+(\Sl(u), \Sl(v))^2\Z\left[\fr{1}{3}\right][[\Sl(u),\Sl(v)]]\ (\subset \Q[[u,v]]).
$$
\end{lemma}
\begin{proof}
After putting $\alpha:=(u+v)/2,\ \beta:=(u-v)/2$, we can see that the partial derivative as $\beta$ of $f(\alpha,\beta):=\fr{\Sl(u)^2\Cl(v)-\Sl(v)^2\Cl(u)}{\Sl(u)\Cl(v)^2-\Sl(v)\Cl(u)^2}$ is zero. Hence, $f$ is a constant function of $\beta$. Therefore $f(\alpha,\beta)=f(\alpha,\alpha)=\Sl(2\alpha)=\Sl(u+v)$, by the last two equalities of Lemma~\ref{pro} \eqref{h}. The equality of Lemma is proved.

It remains to show that 
$
\Sl(u+v)\in \Sl(u)+\Sl(v)+(\Sl(u), \Sl(v))^2\Z\left[\fr{1}{3}\right][[\Sl(u),\Sl(v)]].
$
Since Lemma~\ref{pro} \eqref{h} and \eqref{r},
$$
\Cl(u)=\sqrt[3]{1-\Sl(u)^3}=\sum_{n=0}^{\infty}(-1)^n\dbinom{\fr{1}{3}}{n}\Sl(u)^{3n}=1+\sum_{n=1}^{\infty}(-1)^n\dbinom{\fr{1}{3}}{n}\Sl(u)^{3n},
$$
$$
\Cl(u)^2=\sqrt[3]{(1-\Sl(u)^3)^2}=\sum_{n=0}^{\infty}(-1)^n\dbinom{\fr{2}{3}}{n}\Sl(u)^{3n}=1+\sum_{n=1}^{\infty}(-1)^n\dbinom{\fr{2}{3}}{n}\Sl(u)^{3n}
$$
hold for $u\in \R$ around $u=0$. By the analytic continuations, 
\begin{equation} \label{cl}
\Cl(u)=1+\sum_{n=1}^{\infty}(-1)^n\dbinom{\fr{1}{3}}{n}\Sl(u)^{3n},\ \ 
\Cl(u)^2=1+\sum_{n=1}^{\infty}(-1)^n\dbinom{\fr{2}{3}}{n}\Sl(u)^{3n}
\end{equation}
hold for $u\in \C$ around $u=0$. Write $x=\Sl(u),y=\Sl(v)$. Then, since \eqref{cl},
\begin{equation}
\begin{split} \label{adds}
\Sl(u+v)
&=\fr{x^2-y^2+\sum_{n=1}^{\infty}(-1)^n\dbinom{\fr{1}{3}}{n}(x^2y^{3n}-y^2x^{3n})}{x-y+\sum_{n=1}^{\infty}(-1)^n\dbinom{\fr{2}{3}}{n}(xy^{3n}-yx^{3n})}\\
&=\fr{x+y-x^2y^2\sum_{n=1}^{\infty}(-1)^n\dbinom{\fr{1}{3}}{n}(x^{3n-3}+\cdots+y^{3n-3})}{1-xy\sum_{n=1}^{\infty}(-1)^n\dbinom{\fr{2}{3}}{n}(x^{3n-2}+\cdots+y^{3n-2})}\in x+y+(x, y)^2\Z\left[\fr{1}{3}\right][[x,y]].
\end{split}
\end{equation}
Here, we note that $\left(1-xy\sum_{n=1}^{\infty}(-1)^n\dbinom{\fr{2}{3}}{n}(x^{3n-2}+\cdots+y^{3n-2})\right)^{-1}\in 1+(xy)\Z\left[\fr{1}{3}\right][[x,y]]\subset \left(\Z\left[\fr{1}{3}\right][[x,y]]\right)^{\times}$, and 
$\dbinom{\fr{1}{3}}{n}, \dbinom{\fr{2}{3}}{n}\in \Z\left[\fr{1}{3}\right]$ (See p.52 of \cite{Wa}).

\end{proof}

\begin{lemma} (Power series expansion of $\Sl(ru)$ and $\Sl(ru)^{-1}$ in terms of $\Sl(u)$) \label{addsl}For any integer $r$, we have
$$
\Sl(ru)\in r\Sl(u)+\Sl(u)^4\Z\left[\fr{1}{3}\right][[\Sl(u)^3]]\ (\subset \Q[[u]]).
$$ 
We also have, if $r\neq 0$,
$$\Sl(ru)^{-1}\in \fr{1}{r\Sl(u)}+\Sl(u)^2\Z\left[\fr{1}{r},\fr{1}{3}\right][[\Sl(u)^3]]\ \left(\subset \fr{1}{u}\Q[[u]]\right).$$
\end{lemma}
\begin{proof}
There exists $f(x,y) \in (x,y)^2\Z\left[\fr{1}{3}\right][[x,y]]$ such that 
$$
\Sl(u+v)=\Sl(u)+\Sl(v)+f(\Sl(u),\Sl(v))
$$
from Lemma~\ref{addsl'}.
We will show that, for any $r>0$, $\Sl(ru) \in r\Sl(u)+\Sl(u)^4\Z\left[\fr{1}{3}\right][[\Sl(u)^3]]$ by induction on $r$. We assume that there exists $r>0$ such that 
$
\Sl(ru)\in r\Sl(u)+\Sl(u)^4\Z\left[\fr{1}{3}\right][[\Sl(u)^3]].
$
Then, there exsits $f_{r}(t)\in \Z\left[\fr{1}{3}\right][[t]]$ such that $\Sl(ru)=r\Sl(u)+\Sl(u)^4f_{r}(\Sl(u)^3)$, where $t$ is a variable. Then putting  $v:=ru$, we have 
\begin{align*}
\Sl((r+1)u)&=\Sl(u)+\Sl(ru)+f(\Sl(u),\Sl(ru))\\
&=\Sl(u)+(r\Sl(u)+\Sl(u)^4f_{r}(\Sl(u)^3))+f(\Sl(u),r\Sl(u)+\Sl(u)^4f_{r}(\Sl(u)^3))\\
&\in (r+1)\Sl(u)+\Sl(u)^4\Z\left[\fr{1}{3}\right][[\Sl(u)^3]],
\end{align*}
since 
$f$ has no terms of degree 0 or 1, and the power series expansion of $\Sl((r+1)u)$ as $\Sl(u)$ has only $\Sl(u)^{3n+1}$ terms since $\Sl(\ro u)=\ro \Sl(u)$. Therefore, $\Sl(ru)\in r\Sl(u)+\Sl(u)^4\Z\left[\fr{1}{3}\right][[\Sl(u)^3]]$ is proved for any $r>0$.\\
Next, we will show that $\Sl(-ru) \in -r\Sl(u)+\Sl(u)^4\Z\left[\fr{1}{3}\right][[\Sl(u)^3]]$ for any $r>0$.
Now, since $\Cl(u)^{-1}=\left(1+\sum_{n=1}^{\infty}(-1)^n\dbinom{\fr{1}{3}}{n}\Sl(u)^{3n}\right)^{-1}\in 1+\Sl(u)^3\Z\left[\fr{1}{3}\right][[\Sl(u)^3]]\subset \left(\Z\left[\fr{1}{3}\right][[\Sl(u)^3]]\right)^{\times}$, and Lemma~\ref{pro} \eqref{-},
$$
\Sl(-u)=-\fr{\Sl(u)}{\Cl(u)}=-\fr{\Sl(u)}{1+\sum_{n=1}^{\infty}(-1)^n\dbinom{\fr{1}{3}}{n}\Sl(u)^{3n}}=-\Sl(u)+\cdots \in -\Sl(u)+\Sl(u)^4\Z\left[\fr{1}{3}\right][[\Sl(u)^3]]
$$
follows. Therefore, for any $r>0$, $\Sl(-ru)=r\Sl(-u)+\Sl(-u)^4f_{r}(\Sl(-u)^3)\in -r\Sl(u)+\Sl(u)^4\Z\left[\fr{1}{3}\right][[\Sl(u)^3]]$ follows similarly to the former case. Thus, we have $\Sl(ru)\in r\Sl(u)+\Sl^4(u)\Z\left[\fr{1}{3}\right][[\Sl^3(u)]]$ for any $r\in \Z$.

If $r\neq 0$, we write $\Sl(ru)=r\Sl(u)+\Sl(u)^4f_{r}(\Sl(u)^3)$, where $f_{r}(t)\in \Z\left[\fr{1}{3}\right][[t]]$. Since $\left(1+\fr{\Sl(u)^3f_{r}(\Sl(u)^3)}{r}\right)^{-1}=1+\sum_{n=1}^{\infty}\left(-\fr{\Sl(u)^3f_{r}(\Sl(u)^3)}{r}\right)^n \in 1+\Sl(u)^3\Z\left[\fr{1}{r},\fr{1}{3}\right][[\Sl(u)^3]]$,
\begin{align*}
\Sl(ru)^{-1}&=
\fr{1}{r\Sl(u)}\left(1+\fr{\Sl(u)^3f_{r}(\Sl(u)^3)}{r}\right)^{-1} \in \fr{1}{r\Sl(u)}+\Sl(u)^2\Z\left[\fr{1}{r},\fr{1}{3}\right][[\Sl(u)^3]].
\end{align*}
\end{proof}

\begin{lemma} (An explicit power series expansion of the inverse of $\Sl$)\label{ArcSl}
The function $\Sl:[0,\pe] \to [0,1]$ is the inverse function of
$$
\begin{array}{rcccc}
\ASl: &[0, 1]                     &\longrightarrow& [0, \pe]  & \\
        & \rotatebox{90}{$\in$}&               & \rotatebox{90}{$\in$} & \\
        & t                    & \longmapsto   & \displaystyle \int_{0}^{t}\fr{dt}{\sqrt[3]{(1-t^3)^2}}&\left(=\displaystyle \sum_{n=0}^{\infty}(-1)^n\dbinom{-\fr{2}{3}}{n}\fr{t^{3n+1}}{3n+1}\right),
\end{array}
$$
where we choose the branch of $\sqrt[3]{(-)}$ to be $\in \R$. 
\end{lemma}
\begin{proof}
Since $\fr{d}{dt}\ASl=\fr{1}{\sqrt[3]{(1-t^3)^2}}$, $\ASl$ is monotone increasing. By the definitions of $\ASl$ and $\pe$, $\ASl(0)=0, \ASl(1)=\pe$. Thus, the map $\ASl:[0, 1]\to [0,\pe]$ is bijective. On the other hand, the facts that 
$$
\Sl(0)=0,\ \ \Sl(\pe)=\vp{-\fr{1}{3}}=\fr{6\x{-\fr{\pe}{3}}}{9+\y{-\fr{\pe}{3}}}=1
$$ 
follow by Lemma~\ref{pro} \eqref{phi} and Lemma~\ref{ee} \eqref{val3} \eqref{val3'}. 
By Lemma~\ref{pro} \eqref{d}, $\Sl$ is monotone increasing over $[0,\pe]$. Thus, the map $\Sl:[0,\pe] \to [0,1]$ is a bijection.
Therefore, for $u_{0} \in [0, \pe]$, we write $t_{0}:=\Sl(u_{0})$, then
$$
\ASl(\Sl(u_{0}))=\displaystyle \int_{0}^{t_{0}}\fr{dt}{\sqrt[3]{(1-t^3)^2}}=\int_{0}^{u_{0}}\fr{\Cl(u)^2du}{\sqrt[3]{(1-\Sl(u)^3)^2}}=u_{0},
$$
where $t=\Sl(u)$ and $dt=\Sl'(u)du=\Cl^2(u)du$, since Lemma~\ref{pro} \eqref{d} and \eqref{h}. Note that $\Sl(u), \Cl(u)\in \R$ for $u \in \R$ (Lemma~\ref{pro} \eqref{r}). This means that the map $\Sl:[0,\pe] \to [0,1]$ is the inverse function of the map $\ASl:[0, 1]\to [0,\pe]$.
\end{proof}

From \cite[Theorem 1.19]{Kodaira} and $\Sl'(0)=\Cl^2(0)=1\neq 0$, we have $\Sl(u)$ is bijective around $u=0$ for $u\in \C$. Hence, by Lemma~\ref{ArcSl}, the inverse function of $\Sl(u)$ is  the analytic continuation of the function $\ASl(t)$ to a sufficiently small nighbourhood of $t=0$, which is $\ASl(t)=\sum_{n=0}^{\infty}(-1)^n\dbinom{-\fr{2}{3}}{n}\fr{t^{3n+1}}{3n+1}$. Thus, we have:
\begin{corollary} \label{ASL}
$u=\sum_{n=0}^{\infty}(-1)^n\dbinom{-\fr{2}{3}}{n}\fr{\Sl(u)^{3n+1}}{3n+1}.$
\end{corollary}
Thus, we have $\Q[[u]]=\Q[[\Sl(u)]]$.


\subsection{Bernoulli-Hurwitz-type numbers for the Eisenstein integers}
In this section, we introduce Bernoulli-Hurwitz-type numbers for Eisenstein integers, and study their denominators for the proof of  Theorem~\ref{main1}.

We denote the coefficients of the power series expansion of $\Sl(u)$ around $u=0$ as
$$
\Sl(u)=\sum_{m=0}^{\infty}C_{3m+1}u^{3m+1}=\sum_{m=0}^{\infty}c_{3m+1}\fr{u^{3m+1}}{(3m+1)!}.
$$
(Note that, since $\Sl(\ro u)=\ro \Sl(u)$ by Lemma~\ref{pro} \eqref{r}, the coefficients of $u^n$ in $\Sl(u)$ are zero unless $n$ is congruent $1$ modulo $3$.)

\begin{lemma} (Integrality of $c_{n}$) \label{denc}
We have $c_{3m+1}\in \Z$ for any $m\geq 0$.
\end{lemma}
\begin{proof}
We see that 
\begin{equation} \label{des}
\Sl^{(3)}=6\Sl^4-4\Sl
\end{equation}
 by Lemma~\ref{pro} \eqref{d} and \eqref{h}. Hence, by induction on $m$, we have
$c_{3m+1}\in \Z$ for any $m\geq 0$.
\end{proof}

We denote the coefficients of the power series expansion of $\fr{1}{\Sl(u)}$ around $u=0$ as
\begin{equation} \label{defslinv}
\fr{1}{\Sl(u)}=\fr{1}{u}+\sum_{n=0}^{\infty}D_{3n+2}u^{3n+2}=\fr{1}{u}+\sum_{n=0}^{\infty}\fr{d_{3n+2}}{(3n+2)!}u^{3n+2}.
\end{equation}
We also put $D_{-1}:=1$, then we can express it as $\fr{1}{\Sl(u)}=\sum_{n=0}^{\infty}D_{3n-1}u^{3n-1}$.

We call $C_{3m+1}$, $D_{3n+2}$ {\it Bernoulli-Hurwitz-type numbers} for the Eisenstein integers.

\begin{remark} (Non-integrality of $d_{n}$) \label{remD}The function $f(u)=\fr{1}{\Sl(u)}$ satisfies a differential equation $f^{(3)}=4f-6f^4$, hence $D_{3m+2}\in \Q$ for $m\geq 0$. We have $D_{2}=\fr{1}{6}, D_{5}=-\fr{1}{252}, D_{8}=-\fr{1}{4536}$. Therefore, $d_{3n+2}=(3n+2)!D_{3n+2}$ is not necessarily a rational integer. However, Lemma~\ref{dend} below holds.
\end{remark}

\begin{definition}For a subring $R$ of $\Q$, we define
$$
\mathcal{H}_{R}:=\left\{\sum_{n=0}^{\infty}\fr{a_{n}}{n!}t^n \in \Q[[t]]\ \Big| \ \ a_{n}\in R\ \text{for every $n\geq 0$}\right\}.
$$
It is easy to show that this is a subring of $\Q[[t]]$. 
We call $\mathcal{H}_{R}$'s elements {\it Hurwitz-R-integer series}.
\end{definition}
Note that Lemma~\ref{denc} says $\Sl(u)\in \mathcal{H}_{\Z}$.

\begin{lemma} (Integrality of $2d_{6m+2}$ outside $3$)\label{dend}
For $m\geq 0$, we have $2d_{6m+2}=2(6m+2)!D_{6m+2}\in \Z\left[\fr{1}{3}\right]$.
\end{lemma}

\begin{proof}
We write $f(u):=\fr{1}{\Sl(u)}+\fr{1}{\Sl(-u)}$. Then, $f(u)=2\sum_{m=0}^{\infty}D_{6m+2}u^{6m+2}=2\sum_{m=0}^{\infty}\fr{d_{6m+2}}{(6m+2)!}u^{6m+2}$ by \eqref{defslinv}.

On the other hand, by Lemma~\ref{addsl}, we can write $\fr{1}{\Sl(-u)}=-\fr{1}{\Sl(u)}+g(\Sl(u))$, where $g \in t^2\Z\left[\fr{1}{3}\right][[t^3]]$ ($t$ is a variable). Then, $f(u)=g(\Sl(u))$. Here, by noting that $\Sl(u)=\sum_{n=0}^{\infty}\fr{c_{3n+1}}{(3n+1)!}u^{3n+1}\in \mathcal{H}_{\Z}$ by Lemma~\ref{denc}, and that $\Sl(u)$ does not have constant term as the expansion in terms of $u$, we have $f(u)\in \mathcal{H}_{\Z\left[\fr{1}{3}\right]}$. 

Therefore, by comparing the coefficients, we have $2d_{6m+2}\in \Z\left[\fr{1}{3}\right]$.
\end{proof}
\begin{remark}
\begin{enumerate}[(1)]
\item Moreover, we can show $d_{6m+2}\in \Z\left[\fr{1}{3}\right]$. See Lemma~\ref{dend'} (which is not used to show the main congruences in Theorem~\ref{main1}).

\item For the denominators of $d_{6m-1}$, see Lemma~\ref{den-1}  (which is not used to show the main congruences in Theorem~\ref{main1}).
\end{enumerate}
\end{remark}

\subsection{Eisenstein's product formula}
We recall that the arithmetic properties of $\Sl$ at division points.
\begin{definition} (Primarity)\label{pri}
For $a\in \Z[\ro]$, we call $a$ {\it primary} if $a\equiv 1\bmod 3$.
This definition is slightly different from the ones in \cite[\S 7.1]{Le} ($a\equiv \pm 1 \bmod 3$) and \cite[\S 3 of Chap. 9]{IR} ($a\equiv -1 \bmod 3$).
\end{definition}
Note that, for a prime ideal $\mathfrak{p} (\nmid (3))$, there uniquely exists a primary element $\mu$ such that $\mathfrak{p}=(\mu)$. Let $\ell$ be a rational prime number such that $\ell \equiv 1 \bmod 3$ and fix $\la \in \Z[\ro]$ satisfying
$$
\ell=\la\overline{\la} \ \ \text{and } \la \equiv 1 \bmod 3.
$$
Note that there are precisely two choices of such $\la$, i.e., $\la$ and $\ol{\la}$. 
Then, there is a canonical isomorphism $\Z[\ro]/(\la) \simeq \Z/\ell\Z$.
Let $\chi_{\la}$ be the {\it cubic residue character} to the modulus $\la$ ; the notation is fixed throughout this paper:
$$
\chi_{\la}(\nu):=\left(\fr{\nu}{\la}\right)_{3} ,\ \chi_{\la}(\nu)^3=1, \text{ and } \chi_{\la}(\nu)\equiv \nu^{(\ell-1)/3}\  \bmod \la \ \  (\nu \in (\Z[\ro]/(\la))^{\times}).
$$
We write
$$
\La:=\varphi \left(\fr{1}{\la}\right)=\Sl \left(\fr{-3\pe}{\la}\right).
$$
Note that $\La$ is algebraic and $[\Q(\ro,\La) : \Q(\ro)]=\ell-1$ (see  \cite[Lemma 1.3]{Asai}). Let $\sO_{\la}$ be the ring of integers of $\Q(\ro,\La)$.
We recall the properties of $\vp{\fr{r}{\la}}$'s from \cite{Asai}:
\begin{lemma} (Arithmetic properties of $\varphi$ at division points)\label{Epro} (\cite[Lemma 1.3 of \S 1.1.2.,  Definition 1.11 of \S 1.3.2. with Remark, and Appendix of \S 1 3. (xi)]{Asai}) In the notation above, we have the following:
\begin{enumerate}[(1)]
\item We have $\La \in \sO_{\la}$. Moreover, $\La$ is a prime element in $\sO_{\la}$ and has a property $(\la)=(\La)^{\ell-1}$ as ideals of $\sO_{\la}$. For any $r\not\equiv 0 \bmod \ell$, it holds that $(\varphi(\fr{r}{\la}))=(\La)$ as ideals of $\sO_{\la}$. \label{E1}
\item (Eisenstein's product formula) Let $\widetilde{\la}:=\prod_{\chi_{\la}(r)=1}\vp{\fr{r}{\la}}$. Then $\widetilde{\la}^3=\la$. \label{E2}
\end{enumerate}
\end{lemma}

\begin{lemma}\label{lala}
$$
\widetilde{\la}^2\equiv \La^{\fr{2(\ell-1)}{3}} \bmod \La^{\fr{2(\ell-1)}{3}+1}.
$$
\end{lemma}
\begin{proof}
By Lemma~\ref{addsl}, for $r\in \Z$, we can write 
$$
\vp{\fr{r}{\la}}=r\La+\La^4f_{r}(\La^3)=r\La\left(1+\fr{1}{r}\La^3f_{r}(\La^3)\right),
$$
where $f_{r}(t)\in \Z[\fr{1}{3}][[t]]$, 
($t$ is a  variable).
Let $g\in \Z$ be a primitive root of $1$ modulo $\ell$. Then, since $\ker(\chi_{\la})=\{g^{3j} \bmod \la \mid j=1,\ldots, \fr{\ell-1}{3}\}$ (note that $(\Z[\ro]/(\la))^{\times} \simeq (\Z/\ell\Z)^{\times}$), we have
\begin{align} \label{tla}
\begin{split}
\widetilde{\la}=
\left(\prod_{j=1}^{\fr{\ell-1}{3}}g^{3j}\right)\La^{\fr{\ell-1}{3}}\cdot \prod_{j=1}^{\fr{\ell-1}{3}}\left(1+\fr{1}{g^{3j}}\La^3f_{g^{3j}}(\La^3)\right).
\end{split}
\end{align}
Since $g^{3j}$ is relatively prime to $\ell$ for any $j=1,\ldots,\fr{\ell-1}{3}$,
$$
\prod_{j=1}^{\fr{\ell-1}{3}}\left(1+\fr{1}{g^{3j}}tf_{g^{3j}}(t)\right)\in 1+t\Z\left[\fr{1}{3},\fr{1}{g^{3j}}: j=1,\ldots,\fr{\ell-1}{3}\right][[t]]\subseteq 1+t\Z_{(\ell)}[[t]],
$$
hence,
$$
\prod_{j=1}^{\fr{\ell-1}{3}}\left(1+\fr{1}{g^{3j}}\La^3f_{g^{3j}}(\La^3)\right)
\in \sO_{\la, \La},
$$
where $\sO_{\la, \La}$ is the completion of $\sO_{\la}$ by $\La$.  
Thus, by taking $\bmod \ \La^{\fr{\ell-1}{3}+1}$ of \eqref{tla}, we have  
$$
\widetilde{\la}\equiv \left(\prod_{j=1}^{\fr{\ell-1}{3}}g^{3j}\right)\La^{\fr{\ell-1}{3}} \equiv -\La^{\fr{\ell-1}{3}} \bmod \La^{\fr{\ell-1}{3}+1},
$$
since $\prod_{j=1}^{\fr{\ell-1}{3}}g^{3j}=(g^\fr{\ell-1}{2})^\fr{\ell+2}{3}\equiv -1 \bmod \ell.$
We note that $\fr{\ell+2}{3}$ is odd because $\ell$ is odd prime.
Thus, there exists $a\in \mathscr{O}_{\la}$ such that $\widetilde{\la}=-\La^{\fr{\ell-1}{3}}+a\La^{\fr{\ell-1}{3}+1}=-\La^{\fr{\ell-1}{3}}(1-a\La)$. We have
$$
\widetilde{\la}^2=\La^{\fr{2}{3}(\ell-1)}(1-a\La)^2
\equiv \La^{\fr{2}{3}(\ell-1)}\bmod \La^{\fr{2}{3}(\ell-1)+1}.
$$
\end{proof}

\section{Elliptic Gauss sums} \label{EGs}
In this chapter, we introduce the elliptic Gauss sums and show the main congruences between the Bernoulli-Hurwitz-type numbers and the coefficients of the elliptic Gauss sums. In \S 2.1, we introduce the elliptic Gauss sums and their coefficients due to Asai. In \S 2.2, we show the main congruences and give some examples.

Let $\ell, \la ,\widetilde{\la}, \chi_{\la}$ and $\La$ be as in Chapter~\ref{Ef}. 

\subsection{Elliptic Gauss sums and their coefficients}
In this section, we introduce the elliptic Gauss sums and their coefficients.

We put
$$f:=\left\{
\begin{array}{ll}
\varphi & \text{if\ } \ell \equiv 7 \bmod 9, \\
\varphi^{-1} & \text{if\ } \ell \equiv 4 \bmod 9.
\end{array}
\right.
$$
In this paper, we do not treat the case where $\ell \equiv 1\bmod 9$, since the Hecke $L$-value $L(\widetilde{\chi_{\la}},1)$, which is described by an elliptic Gauss sum (cf. \cite[Theorem 1.16]{Asai}), can vanish (cf. \cite[Remark of Corollary 1.22]{Asai}) in the case where $\ell \equiv 1\bmod 9$. Note that, in the cases where $\ell \equiv 7,4\bmod 9$, the Hecke $L$-value $L(\widetilde{\chi_{\la}},1)$ does not vanish (cf. \cite[Corollary 1.22]{Asai}). 
\begin{definition} (Elliptic Gauss sum) \cite[Definition 1.8]{Asai} We define
\begin{equation}
G_{\la}(\chi_{\la},f):=\fr{1}{3}\sum_{r \in (\Z[\ro]/(\la))^{\times}}\chi_{\la}(r)f\left(\fr{r}{\la}\right),
\end{equation}
and call $G_{\la}(\chi_{\la},f)$ the {\it elliptic Gauss sum} associated to $\chi_{\la}$ and $f$.
\end{definition}

\begin{lemma} \label{EGS}
\begin{enumerate}[(1)]
\item $G_{\la}(\chi_{\la},f)\in \mathscr{O}_{\la}$. \label{Oi}
\item $G_{\la}(\chi_{\la},f)=\sum_{r \in \ker(\chi_{\la})}f\left(\fr{r}{\la}\right)$. \label{k}
\end{enumerate}
\begin{proof}
\eqref{Oi} See \cite[Lemma 1.9]{Asai}. 

\eqref{k} We have $(\Z[\ro]/(\la))^{\times}=S\cup \ro S\cup \rob S$ (disjoint union), where $S:=\ker(\chi_{\la})$, if $\ell \equiv 7 ,4 \bmod 9$.
If $\ell \equiv 7 \bmod 9$ (resp. $\ell \equiv 4 \bmod 9$), then $f=\varphi$ (resp. $f=\varphi^{-1}$), $\vp{\ro u}=\ro \vp{u}$ (resp. $\vp{\ro u}^{-1}=\rob \vp{u}^{-1}$) (by Lemma~\ref{pro} \eqref{r}), and $\chi_{\la}(\ro)=\ro^{\fr{\ell-1}{3}}=\rob$ (resp. $\chi_{\la}(\ro)=\ro^{\fr{\ell-1}{3}}=\ro$) (by the definition of $\chi_{\la}$). Hence, we have
$G_{\la}(\chi_{\la},f)=\fr{1}{3}\sum_{r \in S}\left(f\left(\fr{r}{\la}\right)+\chi_{\la}(\ro)f\left(\fr{\ro r}{\la}\right)+\chi_{\la}(\rob)f\left(\fr{\rob r}{\la}\right) \right)=\sum_{r \in S}f\left(\fr{r}{\la}\right)$.
\end{proof}
\end{lemma}

We shall recall Asai's results from \S 1 of \cite{Asai}:
\begin{theorem} (The descent of $\chi_{\la}(3)^{-1}\al$ or $\ol{\chi_{\la}}(3)^{-1}\al$ and the non-vanishing) \cite[Definition 1.12 and Theorem 1.19]{Asai}\label{defco}
\begin{enumerate}[(1)]
\item \label{7}
If $\ell \equiv 7 \bmod 9$, then there exists $\al \in \chi_{\la}(3)(1+3\Z)$ such that
$$
G_{\la}(\chi_{\la},\varphi)=\al \widetilde{\la}^2.
$$
\item \label{4}
If $\ell \equiv 4 \bmod 9$, then there exists $\al \in \overline{\chi_{\la}}(3)(-1+3\Z)$ such that
$$
G_{\la}(\chi_{\la},\varphi^{-1})=\al \widetilde{\la}^2.
$$
\end{enumerate}
In particular, $\al$ does not vanish in the cases of $\ell \equiv 7, 4 \bmod 9$.
\end{theorem}
\begin{remark}
In $\ell \equiv 1\bmod 9$, it may happen $\al=0$ (cf. \cite[Remark of Corollary 1.22 of \S 1 and Table 1]{Asai}).
\end{remark}
The number $\al$ is called the {\it coefficient of elliptic Gauss sum} in \cite{Asai}.

\subsection{The main congruences}
In this section, we show the main congruences and give some examples.

\begin{theorem} (The congruences between the coefficients of elliptic Gauss sums and the Bernoulli-Hurwitz-type numbers)\label{main1} One has
$$\al\equiv \left\{
\begin{array}{ll}
-\fr{1}{3}C_{\fr{2}{3}(\ell-1)}\bmod \la & \text{if\ }\ell \equiv 7 \bmod 9, \\
-\fr{1}{3}D_{\fr{2}{3}(\ell-1)}\bmod \la & \text{if\ }\ell \equiv 4 \bmod 9.
\end{array}
\right.
$$
\end{theorem}
\begin{proof}
Let $t:=\Sl(u)$. By Corollary~\ref{ASL}, for any $\nu \in \Z[\ro]$, we have
\begin{align} \label{ade}
\Sl(\nu u)&=\sum_{m=0}^{\infty}C_{3m+1}(\nu u)^{3m+1}=\sum_{m=0}^{\infty}C_{3m+1}\nu^{3m+1}\left(\sum_{n=0}^{\infty}(-1)^n\dbinom{-\fr{2}{3}}{n}\fr{t^{3n+1}}{3n+1}\right)^{3m+1}, \\ \label{bde}
\fr{1}{\Sl(\nu u)}+\fr{1}{\Sl(-\nu u)}&=2\sum_{m=0}^{\infty}D_{6m+2}(\nu u)^{6m+2}=2\sum_{m=0}^{\infty}D_{6m+2}\nu^{6m+2}\left(\sum_{n=0}^{\infty}(-1)^n\dbinom{-\fr{2}{3}}{n}\fr{t^{3n+1}}{3n+1}\right)^{6m+2},
\end{align}
as elements of $\Q[[u]]=\Q[[t]]$.

\underline{The case $\ell \equiv 7\bmod 9$.} Let $g\in \Z$ be a primitive root of $1$ modulo $\ell$, i.e., $g \bmod \ell$ generates $(\Z/\ell\Z)^{\times}$. We consider the sum
$$
L_{1}(u):=\Sl(g^3 u)+\Sl(g^6u)+\cdots +\Sl(g^{\ell-1}u).
$$
By Lemma~\ref{EGS} \eqref{k},
$$
G_{\la}(\chi_{\la},\varphi)=\sum_{j=1}^{\fr{\ell-1}{3}}\s{\fr{-3\pe g^{3j}}{\la}}=L_{1}\left(\fr{-3\pe}{\la}\right).
$$
By \eqref{ade},
\begin{align*}
L_{1}(u)=\sum_{m=0}^{\infty}C_{3m+1}\left(\sum_{j=1}^{\fr{\ell-1}{3}}g^{3(3m+1)j}\right)t^{3m+1}\left(\sum_{n=0}^{\infty}(-1)^n\dbinom{-\fr{2}{3}}{n}\fr{t^{3n}}{3n+1}\right)^{3m+1}.
\end{align*}
Since Lemma~\ref{addsl} and the definition of $L_{1}$, the coefficients of this expansion in terms of $t$ are $\in \Z\left[\fr{1}{3}\right]$, especially $\in \Z_{(\ell)}$. Here, we substitute $\La$ for $t$ (which corresponds to the substitution of $\fr{-3\pe}{\la}$ for $u$), then, $\Sl(\nu u)\mapsto \s{\fr{-3\pe\nu }{\la}}$ for any $\nu \in \Z$, since Lemma~\ref{addsl}. Therefore, 
$$
\sum_{m=0}^{\infty}C_{3m+1}\left(\sum_{j=1}^{\fr{\ell-1}{3}}g^{3(3m+1)j}\right)
\La^{3m+1}\left(\sum_{n=0}^{\infty}(-1)^n\dbinom{-\fr{2}{3}}{n}\fr{\La^{3n}}{3n+1}\right)^{3m+1}
$$
converges $\La$-adically in $\sO_{\la, \La}$ and is equal to $G_{\la}(\chi_{\la},\varphi)$. By noting that $(-1)^n\dbinom{-\fr{2}{3}}{n}\fr{1}{3n+1}\in \Z_{(\ell)}$ for $n \leq \fr{2\ell-5}{9}$ (since $\dbinom{-\fr{2}{3}}{n}\in \Z\left[\fr{1}{3}\right]$ and $3n+1\leq \fr{2\ell-2}{3}<\ell$) and $C_{3m+1}\in \Z_{(\ell)}$ for $m \leq \fr{2\ell-5}{9}$ (since $C_{3m+1}=\fr{c_{3m+1}}{(3m+1)!}$,  $3m+1<\ell$, and $c_{3m+1}\in \Z$ (Lemma~\ref{denc})), and that  the coefficients of this expansion in terms of $\La$ are $\in \Z\left[\fr{1}{3}\right]$,
we have
$$
G_{\la}(\chi_{\la},\varphi)\equiv \sum_{m=0}^{\fr{2\ell-5}{9}}C_{3m+1}\left(\sum_{j=1}^{\fr{\ell-1}{3}}g^{3(3m+1)j}\right)\La^{3m+1}\left(\sum_{n=0}^{\fr{2\ell-5}{9}}(-1)^n\dbinom{-\fr{2}{3}}{n}\fr{\La^{3n}}{3n+1}\right)^{3m+1} \bmod \La^{\fr{2(\ell-1)}{3}+1}.
$$
We also have
$$
\sum_{j=1}^{\fr{\ell-1}{3}}g^{3(3m+1)j}\equiv \left\{
\begin{array}{lll}
0 & \bmod \ell & \text{if\ } \ell-1\mathrel{\not|}3(3m+1), \\
\fr{\ell-1}{3} & \bmod \ell & \text{if\ } \ell-1|3(3m+1).
\end{array}
\right.
$$
Since $0\leq m \leq \fr{2\ell-5}{9}$ and $\ell \equiv 7 \bmod 9$, we see that $\ell-1|3(3m+1)$ if and only if $3(3m+1)=2(\ell-1)$. Then, we have
$$
G_{\la}(\chi_{\la},\varphi)\equiv C_{\fr{2(\ell-1)}{3}}\cdot \fr{\ell-1}{3}\cdot \La^{\fr{2(\ell-1)}{3}}=-\fr{1}{3}C_{\fr{2(\ell-1)}{3}}\cdot \La^{\fr{2(\ell-1)}{3}}\bmod \La^{\fr{2(\ell-1)}{3}+1}.
$$
On the other hand, by Theorem~\ref{defco} \eqref{7} and Lemma~\ref{lala}, we have 
$$
G_{\la}(\chi_{\la},\varphi)=\al \widetilde{\la}^2\equiv \al \La^{\fr{2(\ell-1)}{3}}\bmod \La^{\fr{2(\ell-1)}{3}+1}.
$$
Therefore,
$$
\al \equiv -\fr{1}{3}C_{\fr{2(\ell-1)}{3}}\bmod \La
$$
holds. Since $\al \in \chi_{\la}(3)(1+3\Z) \subset \Z[\ro]$, and $C_{\fr{2(\ell-1)}{3}}=c_{\fr{2(\ell-1)}{3}}/(\fr{2(\ell-1)}{3})! \in \Z_{(\ell)}$ by Theorem~\ref{defco} \eqref{7} and Lemma~\ref{denc}, we have $\al \equiv -\fr{1}{3}C_{\fr{2(\ell-1)}{3}}\bmod \la$.

\underline{The case $\ell \equiv 4\bmod 9$.} Let $g\in \Z$ be a primitive root of $1$ modulo $\ell$. We consider the sum
\begin{equation}
L_{2}(u):=\Sl(g^3u)^{-1}+\Sl(g^6 u)^{-1}+\cdots +\Sl(g^{\fr{\ell-1}{2}}u)^{-1}+\Sl(-g^3u)^{-1}+\Sl(-g^6 u)^{-1}+\cdots +\Sl(-g^{\fr{\ell-1}{2}}u)^{-1}.
\end{equation}
By noting that $g^{\fr{\ell-1}{2}}\equiv -1\bmod \ell$, we have $\ker(\chi_{\la})=\{g^{3j}, -g^{3j} \bmod \ell \mid j=1,\ldots, \fr{\ell-1}{6}\}$. By this fact and Lemma~\ref{EGS} \eqref{k}, we have
$$
G_{\la}(\chi_{\la},\varphi^{-1})=\sum_{j=1}^{\fr{\ell-1}{6}}\left(\s{\fr{-3\pe g^{3j}}{\la}}^{-1}+ \s{\fr{3\pe g^{3j}}{\la}}^{-1}\right)=L_{2}\left(\fr{-3\pe}{\la}\right).
$$
By \eqref{bde},
\begin{align*}
L_{2}(u)&=2\sum_{m=0}^{\infty}D_{6m+2}\left(\sum_{j=1}^{\fr{\ell-1}{6}}g^{3(6m+2)j}\right)t^{6m+2}\left(\sum_{n=0}^{\infty}(-1)^n\dbinom{-\fr{2}{3}}{n}\fr{t^{3n}}{3n+1}\right)^{6m+2}.
\end{align*}
Since Lemma~\ref{addsl} and the definition of $L_{2}$, the coefficients of this expansion in terms of $t$ are $\in \Z\left[\fr{1}{g^{3j}},\fr{1}{3}: j=1,\ldots,\fr{\ell-1}{6}\right]$, especially $\in \Z_{(\ell)}$. Here, we substitute $\La$ for $t$ (which corresponds to the substitution of $\fr{-3\pe}{\la}$ for $u$), then, $\Sl(\nu u)\mapsto \s{\fr{-3\pe\nu }{\la}}$ for any $\nu \in \Z$, since Lemma~\ref{addsl}. Therefore, 
$$
2\sum_{m=0}^{\infty}D_{6m+2}\left(\sum_{j=1}^{\fr{\ell-1}{6}}g^{3(6m+2)j}\right)\La^{6m+2}\left(\sum_{n=0}^{\infty}(-1)^n\dbinom{-\fr{2}{3}}{n}\fr{\La^{3n}}{3n+1}\right)^{6m+2}
$$
converges $\La$-adically in $\sO_{\la,\La}$ and is equal to $G_{\la}(\chi_{\la},\varphi^{-1})$. By noting that $(-1)^n\dbinom{-\fr{2}{3}}{n}\fr{1}{3n+1}\in \Z_{(\ell)}$ for $n \leq \fr{2(\ell-4)}{9}$ (since $\dbinom{-\fr{2}{3}}{n}\in \Z\left[\fr{1}{3}\right]$ and $3n+1\leq \fr{2\ell-3}{3}<\ell$) and $D_{6m+2}\in \Z_{(\ell)}$ for $m \leq \fr{\ell-4}{9}$ (since $D_{6m+2}=\fr{d_{6m+2}}{(6m+2)!}$,  $6m+2<\ell$, and $d_{6m+2}\in \Z_{(\ell)}$ (Lemma~\ref{dend})), and that  the coefficients of this expansion in terms of $\La$ are $\in \Z\left[\fr{1}{3}\right]$, we have
$$
G_{\la}(\chi_{\la},\varphi^{-1})\equiv 2\sum_{m=0}^{\fr{\ell-4}{9}}D_{6m+2}\left(\sum_{j=1}^{\fr{\ell-1}{6}}g^{3(6m+2)j}\right)\La^{6m+2}\left(\sum_{n=0}^{\fr{2(\ell-4)}{9}}(-1)^n\dbinom{-\fr{2}{3}}{n}\fr{\La^{3n}}{3n+1}\right)^{6m+2} \bmod \La^{\fr{2(\ell-1)}{3}+1}.
$$
We also have
$$
\sum_{j=1}^{\fr{\ell-1}{6}}g^{3(6m+2)j}\equiv \left\{
\begin{array}{lll}
0 & \bmod \ell & \text{if\ } \ell-1\mathrel{\not|}3(6m+2), \\
\fr{\ell-1}{6} & \bmod \ell & \text{if\ } \ell-1|3(6m+2).
\end{array}
\right.
$$
Since $0\leq m \leq \fr{\ell-4}{9}$ and $\ell \equiv 4 \bmod 9$, we see that $\ell-1|3(6m+2)$ if and only if $3(6m+2)=2(\ell-1)$. Then, we have
$$
G_{\la}(\chi_{\la},\varphi^{-1})\equiv 2D_{\fr{2(\ell-1)}{3}}\cdot \fr{\ell-1}{6}\cdot \La^{\fr{2(\ell-1)}{3}}=-\fr{1}{3}D_{\fr{2(\ell-1)}{3}}\La^{\fr{2(\ell-1)}{3}} \bmod \La^{\fr{2(\ell-1)}{3}+1}.
$$
On the other hand, by Theorem~\ref{defco} \eqref{4} and Lemma~\ref{lala}, we have

$$
G_{\la}(\chi_{\la},\varphi^{-1})=\al \widetilde{\la}^2\equiv \al \La^{\fr{2(\ell-1)}{3}} \bmod \La^{\fr{2(\ell-1)}{3}+1}.
$$
Therefore, 
$$
\al \equiv -\fr{1}{3}D_{\fr{2(\ell-1)}{3}}\bmod \La
$$
holds. 

Since $\al \in \ol{\chi_{\la}}(3)(-1+3\Z) \subset \Z[\ro]$, and $D_{\fr{2(\ell-1)}{3}}=d_{\fr{2(\ell-1)}{3}}/(\fr{2(\ell-1)}{3})! \in \Z_{(\ell)}$ by Theorem~\ref{defco} \eqref{4} and Lemma~\ref{dend}, we have $\al \equiv -\fr{1}{3}D_{\fr{2(\ell-1)}{3}}\bmod \la$.
\end{proof}


\begin{example}
\begin{enumerate}[(1)]
\item Let $\ell=7$ and $\la=1+3\ro$. By Table~\ref{tc}, we have $C_{4}=-\fr{1}{6}$.
Then,
$$
-\fr{1}{3}C_{\fr{2(\ell-1)}{3}}=-\fr{1}{3}C_{4}=\fr{1}{18}\equiv 2 \bmod 7.
$$
By noting that $\Z[\ro]/(\la) \simeq \Z/\ell\Z: \ro \mapsto 2$, this coincides with the value $\al=1\cdot \ro$ in the table at \cite[Table 1. in p.117]{Asai}.

\item $\ell=13$ and $\la=4+3\ro$. By Table~\ref{td}, $D_{8}=-\fr{1}{4536}\equiv 1 \bmod 13$. Then,
$$
-\fr{1}{3}D_{\fr{2(\ell-1)}{3}}=-\fr{1}{3}D_{8}\equiv -\fr{1}{3}\equiv -3^2 \bmod 13.
$$
By noting that $\Z[\ro]/(\la) \simeq \Z/\ell\Z:\ro \mapsto 3$, this coincides with the value $\al=-1\cdot \rob$ in the table at \cite[Table 1. in p.117]{Asai}.
\end{enumerate}
\end{example}

\begin{lemma} (The congruences between the square of the absolute value of the coefficients of the elliptic Gauss sums and the Bernoulli-Hurwitz-type numbers)\label{ab}
We have
$$
|\al|^2\underset{\bmod \la}{\equiv} \left\{
\begin{array}{ll}
\left(\ol{\chi_{\la}(3)}C_{\fr{2(\ell-1)}{3}}\right)^2/9 & \text{if\ } \ell\equiv 7 \bmod 9, \\
\left(\chi_{\la}(3)D_{\fr{2(\ell-1)}{3}}\right)^2/9 & \text{if\ } \ell\equiv 4 \bmod 9.
\end{array}
\right.
$$
We also have
$$
|\al|^2\underset{\bmod \ell}{\equiv} \left\{
\begin{array}{ll}
\left(3^{\fr{\ell-7}{6}}C_{\fr{2(\ell-1)}{3}}\right)^2 & \text{if\ } \ell\equiv 7 \bmod 9, \\
\left(3^{\fr{\ell-4}{3}}D_{\fr{2(\ell-1)}{3}}\right)^2 & \text{if\ } \ell\equiv 4 \bmod 9.
\end{array}
\right.
$$
\end{lemma}

\begin{proof}
Let $\ell \equiv 7\bmod 9$.
We have
$\al^2 \equiv \left(C_{\fr{2(\ell-1)}{3}}\right)^2/9 \bmod \la$ by Theorem~\ref{main1}. We can take $a_{\la}\in 1+3\Z$ as 
$\al=\chi_{\la}(3)a_{\la}$ by Theorem~\ref{defco} \eqref{7}. Then, $\al^2=\ol{\chi_{\la}(3)}a_{\la}^2$, and $|\al|^2=a_{\la}^2$ hold. Therefore, we have
$$
|\al|^2=a_{\la}^2=\chi_{\la}(3)\al^2\equiv \left(\ol{\chi_{\la}(3)}C_{\fr{2(\ell-1)}{3}}\right)^2/9 \bmod \la.
$$

Since $\ol{\chi_{\la}(3)}^2=\chi_{\la}(3)\equiv 3^{\fr{\ell-1}{3}}=9\cdot (3^{\fr{\ell-7}{6}})^2 \bmod \la$ and $|\al|^2, \left(3^{\fr{\ell-7}{6}}C_{\fr{2(\ell-1)}{3}}\right)^2 \in \Q$, the second congruence in the case where $\ell \equiv 7\bmod 9$ is proved. 

The case where $\ell \equiv 4\bmod 9$ is proved similarly to the case where $\ell \equiv 7\bmod 9$.
\end{proof}
\begin{remark}
It seems that the statement of \cite[Corollary 2.16]{O}, which is the counterpart of Lemma~\ref{ab} for Gaussian integers case, is a typo of
$$
|\al|^2 \underset{\bmod \ell}{\equiv}  
\begin{cases}
\left(\fr{1}{4}C_{\fr{3(\ell-1)}{4}}\right)^2 & \ell\equiv 13 \bmod 16, \\
-\left(\fr{1}{4}D_{\fr{3(\ell-1)}{4}}\right)^2 & \ell\equiv 5 \bmod 16.
\end{cases}
$$
\end{remark}

\section{Elliptic Gauss sums and associated elliptic curves}
In this chapter, we study certain Hecke characters, Hecke $L$-series associated to them, certain elliptic curves $\E_{\la}$ for $\la$, and their $L$-functions relevant to our elliptic Gauss sums.

In \S 3.1, we define appropriate Hecke characters and the Hecke $L$-series associated to them. In \S 3.2, we introduce certain elliptic curves $\E_{\la}$ for $\la$ and show that their $L$-functions correspond to the Hecke $L$-series defined in \S 3.1. In \S 3.3, by using the coefficient of the elliptic Gauss sum, we show that the order of the Tate-Shafarevich group of $\E_{\la}/\Q(\ro)$ is congruent to the Bernoulli-Hurwitz-type number modulo $\ell$ under the BSD conjecture on the $2, 3$-parts of the leading term. We also obtain unconditional results.  

Let $\ell, \la ,\widetilde{\la}, \chi_{\la}, W$, and $W'$ be as in Chapter~\ref{Ef}.

\subsection{Hecke $L$-series}
In this section, we introduce appropriate Hecke characters and the Hecke $L$-series associated to them.

We define, due to \cite[\S 1.4.1]{Asai},
$$\widetilde{\chi_{\la}}((\nu)):=\chi_{1}(\nu)\ol{\nu},\ \ \chi_{1}:=\left\{
\begin{array}{ll}
\chi_{\la}\cdot\ol{\chi_{0}} & \text{for\ } \ell=\la\ol{\la}\equiv 7 \bmod 9,\\
\chi_{\la}\cdot\chi_{0}' & \text{for\ } \ell=\la\ol{\la}\equiv 4 \bmod 9,
\end{array}
\right.
$$
where $\chi_{0}:(\Z[\ro]/(3))^{\times}\stackrel{\sim}{\longrightarrow} W$ and 
 $\chi_{0}':(\Z[\ro]/(\sqrt{-3}))^{\times}\stackrel{\sim}{\longrightarrow} \{\pm1\}$ are natural isomorphisms.
Note that $\widetilde{\chi_{\la}}((\nu))$ is independent of choice of the generator $\nu$ of the ideal $(\nu)$.
Hecke $L$-series associated to the Hecke character $\widetilde{\chi_{\la}}$ is
\begin{align}
\begin{split} \label{defl} 
L(s,\widetilde{\chi_{\la}})&=\prod_{(\mu):\ \text{prime ideal}}(1-\widetilde{\chi_{\la}}((\mu))(\mu\overline{\mu})^{-s})^{-1}=\prod_{\substack{\mu:\ \text{prime} \\ \mu \equiv 1 \bmod 3\\ \text{or} \ \mu=1-\ro}}(1-\widetilde{\chi_{\la}}((\mu))(\mu\overline{\mu})^{-s})^{-1}\\
&=\prod_{\substack{\mu:\ \text{primary prime} \\ \mu \neq \la, 1-\ro}}(1-\chi_{\la}(\mu)\ol{\mu}(\mu\overline{\mu})^{-s})^{-1}
\end{split}
\end{align}
for $\ell \equiv 7,4 \bmod 9$, since $\chi_{\la}(\la)=0, \chi_{0}(1-\ro)=\chi_{0}'(1-\ro)=0$.

\subsection{$L$-functions of elliptic curves}
In this section, we introduce certain elliptic curves for $\la$, and show the properties of them such as their reduction types at prime ideals of $\Z[\ro]$, the numbers of the points of the reductions of the elliptic curves.
We also show that the Hecke $L$-series defined in \S 3.1 correspond to the $L$-functions of the elliptic curves defined in this section.

First, we recall the properties of Jacobi sums and the cubic reciprocity law which will be used to count points of elliptic curve over a finite field.
\begin{lemma} \label{Ja}
\begin{enumerate}[(1)]
\item (Jacobi sums over $\F_{p}$)\label{pp}
For any primary prime $\mu\in \Z[\ro]$ such that $p=\mu\ol{\mu}\equiv 1\bmod 3$ is a rational prime number, and any characters $\phi, \chi$ of $\F_{p}^{\times}$, we define the {\it Jacobi sum} associated to $\phi, \chi$ to be
$$
J(\phi,\chi):=\sum_{t\in \F_{p}}\phi(t)\chi(1-t).
$$
If $\phi$ is the quadratic residue character, $\chi$ is the cubic residue character, then 
$$
J(\phi, \chi)=\chi(4)J(\chi,\chi)=-\chi(4)\mu
$$
holds.
\item (Jacobi sums over $\F_{q^2}$)\label{qq}
For any primary prime $-q \in \Z[\ro]$ such that $q$ is a rational prime number (automatically satisfying $q \equiv 2\bmod 3$), and any characters $\phi, \chi$ of $\F_{q^2}^{\times}$, we define
$$
J_{2}(\phi,\chi):=\sum_{t\in \F_{q^2}}\phi(t)\chi(1-t).
$$
If $\phi$ has order 2 and $\chi$ has order 3, then 
$$
J_{2}(\phi, \chi)=q
$$
holds.

\item (Cubic Reciprocity Law) \label{rec}
If $\la_{1}, \la_{2}$ are primary primes in $\Z[\ro]$, then $$
\left(\fr{\la_{1}}{\la_{2}}\right)_{3}=\left(\fr{\la_{2}}{\la_{1}}\right)_{3}.
$$
\end{enumerate}
\end{lemma}
\begin{proof}
\eqref{pp} The first equality follows from \cite[Lemma in \S 3 of Chap. 18]{IR} and the second equality follows from \cite[Lemma 1 in \S 4 of Chap. 9]{IR}. (Note that $-\mu$ is primary  in the definition of \cite[\S 3 of Chap. 9]{IR}.)

\eqref{qq} See Theorem 2.3 and Theorem 2.14 in \cite{BE}.

\eqref{rec} See \cite[Theorem 7.8]{Le}. (Note that, the primarity in Definition~\ref{pri} implies the primarity in \cite[\S 7.1]{Le}.)
\end{proof}

Let
$$
\E_{\la}:y^2=x^3+\fr{\la^2}{4}
$$
be an elliptic curve over $\Q(\ro)$, which is a twist of the elliptic curve $y^2=4x^3-27$. This has an automorphism $(x,y)\mapsto (\ro x,-y)$, and we see that 
$$
{\rm End}(\E_{\la})\simeq \Z[\ro].
$$

\begin{lemma} \label{E}
For any prime ideal $\p \subset \Z[\ro]$, 
let $R_{\p}$ be the completion of $\Z[\ro]$ by $\p$, $K_{\p}$ the fraction field of $R_{\p}$, and $k_{\p}(=\Z[\ro]/\p)$ the residue field of $K_{\p}$.
\begin{enumerate}[(1)]
\item (Reductions) The elliptic curve $\E_{\la}$ has bad reductions only at $(\la), (1-\ro)$, moreover, $\E_{\la}$ has additive reductions at $(\la), (1-\ro)$. \label{red}
\item (Special Fibers) The special fiber of $\E_{\la}$ at $(\la)$ is of Type IV, and the one at $(1-\ro)$ is of Type IV if $\ell \equiv 7\bmod 9$, of Type  {\rm I}$_{0}^{*}$ if $\ell \equiv 4\bmod 9$, in the notation of Kodaira symbols in the Kodaira-N\'eron classification of the special fibers of N\'eron models. \label{type}
\item (Local Tamagawa Numbers)
For any prime ideal $\p \subset \Z[\ro]$, let $\mathcal{E}/R_{\p}$ be the N\'eron model of $\E_{\la}\times_{\Q(\ro)}K_{\p}$, 
$\widetilde{\mathcal{E}}/k_{\p}=\mathcal{E}\times_{R_{\p}}k_{\p}$ the special fiber of $\mathcal{E}$, 
$\widetilde{\mathcal{E}^{0}}/k_{\p}$ the identity component of the group variety $\widetilde{\mathcal{E}}/k_{\p}$, 
and $\tau_{\p}$ the order of the group of components $\widetilde{\mathcal{E}}(k_{\p})/\widetilde{\mathcal{E}^{0}}(k_{\p})$. Then, we have 
$$
\tau_{(\la)}=3,\ \tau_{(1-\ro)}=\left\{
\begin{array}{ll}
4 & \text{if}\  n\equiv 0\bmod 3, \\
1 & \text{if otherwise,}
\end{array}
\right.
$$
where $\la=3m+1+3n\ro, m,n \in \Z$. \label{type1}
\item (Point Counting at Good Ordinary Reductions) For a primary prime $\mu \in \Z[\ro]$ such that $p=\mu\ol{\mu}$ is a rational prime number and $\mu \neq \la$, we have
$$
\#(\E_{\la}\otimes \F_{p})(\F_{p})=p+1-(\chi_{\la}(\mu)\ol{\mu}+\ol{\chi_{\la}(\mu)}\mu),
$$ 
where $\E_{\la}\otimes \F_{p}$ is the special fiber of the minimal Weierstrass model of $\E_{\la}$ over $R_{(\mu)}$.
\label{op}
\item (Point Counting at Good Supersingular Reductions) For a rational prime number $q\equiv 2 \bmod3$, we have
$$
\#(\E_{\la}\otimes \F_{q^2})(\F_{q^2})=q^2+1+q(\chi_{\la}(q)+\ol{\chi_{\la}(q)}),
$$ 
where $\E_{\la}\otimes \F_{q^2}$ is the minimal Weierstrass model of $\E_{\la}$ over $R_{(-q)}$.
\label{oq}
\item (Global Rational Torsion Points) 
$$
\E_{\la}(\Q(\ro))_{\text{tors}}=\E_{\la}(\Q(\ro))[3]=\left\{\infty, \left(0,\pm\fr{\la}{2}\right)\right\},
$$
thus $\#\E_{\la}(\Q(\ro))_{\text{tors}}=3.$ Here,  
$\E_{\la}(\Q(\ro))[m], \E_{\la}(\Q(\ro))_{\text{tors}}$ are the $m$-torsion subgroup and the torsion subgroup of $\E_{\la}(\Q(\ro))$, respectively.
 \label{rat}
\end{enumerate}
\end{lemma}
\begin{remark} On \eqref{rat}, we can show, moreover, $\E_{\la}(\Q(\ro))=\E_{\la}(\Q(\ro))_{\text{tors}}$ as follows:
As we will see later, the $L$-function of $\E_{\la}$ does not vanish at $s=1$. This fact follows from Lemma~\ref{L1}, which is proved without Lemma~\ref{E} \eqref{rat}. From \cite[Theorem 1]{CW}, the rank of $\E_{\la}(\Q(\ro))$ is zero. Thus, $\E_{\la}(\Q(\ro))=\E_{\la}(\Q(\ro))_{\text{tors}}$.
\end{remark}

\begin{proof}
\eqref{red}
The substitution $y\mapsto y+\fr{\la}{2}$ to $y^2=x^3+\fr{\la^2}{4}$ gives an equation of the form
$$
\E_{\la}:y^2+\la y=x^3.
$$
This equation has the discriminant $\Delta=-27\la^4$ and $c_{4}=0$, where $c_{4}$ is the one defined in \cite[Chap. III. 1]{Sil'}. Thus $(\Delta)=(1-\ro)^6(\la)^4$ as ideals in $\Z[\ro]$. This is a global minimal Weierstrass equation of $\E_{\la}/\Q(\ro)$, since \cite[Chap. VII, Remark 1.1]{Sil'}. (Note that Weierstrass equation $y^2=x^3+\fr{\la^2}{4}$ is minimal at every prime except $(2)$.)
Hence, $\E_{\la}$ has only two bad primes $(\la), (1-\ro)$, and it has additive reductions at these primes, since \cite[Chap. VII, Proposition 5.1(c)]{Sil'} 
(The fact that $\E_{\la}$ has additive reductions at $(\la), (1-\ro)$ can be deduced in the following way as well: Since $\E_{\la}$ has complex multiplication, the $j$-invariant of $\E_{\la}$ is integral by \cite[Chap. II, Theorem 6.1]{Sil}. Since \cite[Chap. VII, Proposition 5.5]{Sil'}, $\E_{\la}$ has potential good reduction at every primes. By \cite[Chap. VII, Proposition 5.4 (b)]{Sil'}, $\E_{\la}$ does not have multiplicative reductions.).

\eqref{type}-\eqref{type1} In the following, we use the symbols (e.g., $a_{1}, b_{2}, a_{3,1}$ etc.) and the terminology (e.g., Step1 etc.) of  \cite[Chap. IV, Tate's Algorithm 9.4]{Sil}. Note that $a_{1}=0, a_{3}=\la, a_{2}=a_{4}=a_{6}=0, b_{2}=b_{4}=0, b_{6}=\la^2, b_{8}=0$. First, we determine the reduction type at $(\la)$ and $\tau_{(\la)}$. Note that $k_{(\la)}\simeq \F_{\ell}$, and choose $\la$ as a uniformizer. 

Step1: Since $\la|\Delta$, go to Step2.

Step2: Since $\la|a_{3},a_{4},a_{6}$ and $b_{2}$, go to Step3.

Step3: Since $\la^2|a_{6}$, go to Step4.

Step4: Since $\la^3|b_{8}$, go to Step5.

Step5: Since $\la^3\mathrel{\not|} b_{6}=\la^2$, $\E_{\la}$ is of Type IV at $(\la)$.
Let $k'$ be the splitting field over $k_{(\la)}$ of $T^2+a_{3,1}T-a_{6,2}=T^2+T=0$.
Then, $k'=k_{(\la)}$. Thus, we have $\tau_{(\la)}=c=3$.

Next, we determine the reduction type at $(1-\ro)$ and $\tau_{(1-\ro)}$. Note that $k_{(1-\ro)}\simeq\F_{3}$, and choose $1-\ro$ as a uniformizer.

Step1: Since $(1-\ro)|\Delta$, go to Step2.

Step2: $\E_{\la}\otimes k_{(1-\ro)}/k_{(1-\ro)}:y^2+y=x^3$, since $\la\equiv 1\mod 3$, especially, $\la\equiv 1\bmod (1-\ro)$. Here, $(x,y)=(-1,1)$ is the singular point. To make a change of variables to move the singular point to $(0,0)$, we substitute $(x,y)\mapsto (x-1,y+\la)$. Then, $\E_{\la}:y^2+3\la y=x^3-3x^2+3x-1-2\la^2$. The new $a_{i}$'s and $b_{i}$'s are:
$$
a_{1}=0,\ a_{3}=3\la,\ a_{2}=-3,\ a_{4}=3,\ a_{6}=-1-2\la^2,
$$
$$
b_{2}=-12,\ b_{4}=6,\ b_{6}=\la^2-4,\ b_{8}=-3(\la^2-1).
$$
Thus, $(1-\ro)|a_{3},a_{4},a_{6}$ and $b_{2}$, go to Step3.

Step3: Since $a_{6}=-1-2\la^2\equiv -1-2\equiv 0\bmod 3$, $(3)=(1-\ro)^2|a_{6}$, go to Step4.

Step4: Since $\la^2-1\equiv 0\bmod 3$, $(3)^2|b_{8}$. Especially $(1-\ro)^3|b_{8}$, go to Step5.

Step5:  We write $\la=a+b\ro$, $a,b \in \Z$. Since $\la$ is primary, we can write $a=3m+1, b=3n$, $n,m \in \Z$. Then, $\la+2=3(m+1)+3n\ro=3(m+1+n\ro)$, and $\ell= \la\ol{\la}=a^2-ab +b^2=(3m+1)^2-(3m+1)3n+9n^2\equiv 6m+1-3n \bmod 9$. We note that, for $\alpha+\beta\ro \in \Z[\ro]$,
\begin{equation} \label{mod}
(1-\ro)|(\alpha+\beta \ro) \Longleftrightarrow \alpha+\beta \equiv 0\bmod 3.
\end{equation}
Note also that $a_{6}=-1-2\la^2=-1-2(3m+1+3n\ro)^2\equiv -3(1+m+n\ro) \bmod 9$.

\underline{We assume that $\ell \equiv 7 \bmod 9$}. We will show that $(1-\ro)^3\mathrel{\not|} b_{6}$. Note that $b_{6}=(\la+2)(\la-2)$. Since $\la-2\equiv -1\bmod 3$, $(1-\ro)\mathrel{\not|}(\la-2)$. 
Thus, it  suffices to show that $(1-\ro)^3\mathrel{\not|} \la+2$. Therefore, it suffices to show that $(1-\ro)\mathrel{\not|} m+1+n\ro$. Now, $7\equiv \ell\equiv 6m+1-3n \bmod 9$, thus, $-3\equiv -3(m+n) \bmod 9$. Hence, we have $1\equiv m+n \bmod 3$.
Since $m+n+1\equiv 2\bmod 3$ and \eqref{mod}, we have $(1-\ro)\mathrel{\not|} m+1+n\ro$. Therefore, we proved $(1-\ro)^3\mathrel{\not|} b_{6}$.
Thus, $\E_{\la}$ is of Type IV at $(1-\ro)$. We have $a_{6,2}=\fr{a_{6}}{(1-\ro)^2}\equiv \rob(1+m+n\ro)\equiv 1+m+n \equiv -1\bmod (1-\ro)$ (note that $(1-\ro)^2=-3\ro$).
Let $k'$ be the splitting field over $k_{(1-\ro)}$ of $T^2+a_{3,1}T-a_{6,2}=T^2+1\in k_{(1-\ro)}[T]$. 
Then, $k'\neq k_{(1-\ro)}$. Thus, we have $\tau_{(1-\ro)}=c=1$.

\underline{We assume that $\ell \equiv 4 \bmod 9$}. Then, $4\equiv \ell \equiv 6m+1-3n \bmod9$, we have $m+n\equiv -1\bmod 3$. Thus, $(1-\ro)|m+1+n\ro$ by \eqref{mod}. Hence, $(1-\ro)^3|3(m+1+n\ro)=\la+2$. Since $(1-\ro)^3|(\la+2)(\la-2)=b_{6}$, go to Step6.

Step6: Already, $(1-\ro)|a_{1}$, $a_{2}$; $(1-\ro)^2|a_{3}$, $a_{4}$; and $(1-\ro)^3|a_{6}$ are satified. Since $a_{6,3}=\fr{a_{6}}{(1-\ro)^3}\equiv \fr{1+m+n\ro}{\ro(1-\ro)}=\ro^2\fr{1+m+n}{1-\ro}-\ro^2n\equiv -n \bmod 1-\ro$. We consider $P(T)=T^3+a_{2,1}T^2+a_{4,2}T+a_{6,3}=T^3-T-n \in k_{(1-\ro)}[T]$. Since $P'(T)=-1$, the polynomial $P(T)$ has distinct roots in an algebraic closure of $k_{(1-\ro)}$. Thus, $\E_{\la}$ is of Type {\rm I}$_{0}^{*}$ at $(1-\ro)$.
We have
$$
\tau_{(1-\ro)}=c=1+\#\{\alpha \in k_{(1-\ro)}\mid P(\alpha)=0\}=\left\{
\begin{array}{ll}
4 & \text{if}\  n\equiv 0\bmod 3, \\
1 & \text{if otherwise,}
\end{array}
\right.
$$
since $t^3-t=0$ for any $t\in k_{(1-\ro)}\simeq \F_{3}$.

\eqref{op} $y^2=x^3+\fr{\la^2}{4}$ is the minimal Weierstrass equation at the prime ideal $(\mu)$.
Let $\phi: \F_{p}^{\times}\rightarrow \{\pm1\}$ be the quadratic residue symbol, and $\chi: \F_{p}^{\times}\rightarrow W'$ be the cubic residue symbol. Let $D_1\in \F_{p}$ be the image of $\fr{\la^2}{4}$ in $\Z[\ro]/(\mu)\simeq \F_{p}$. Note that $D_1\neq 0, \phi(D_1)=1, \chi(-1)=1$. Then, we have
\begin{align*}
\#(\E_{\la}\otimes \F_{p})(\F_{p})&=1+\#\{(x,y) \in  \F_{p}\times \F_{p} \mid y^2=x^3+D_1\}=1+\sum_{\substack{u,v\in \F_{p} \\ u+v=D_1}}\#\{y\in \F_{p}\mid y^2=u\}\cdot\#\{x\in \F_{p}\mid x^3=-v\}\\
&=1+\sum_{\substack{u,v\in \F_{p} \\ u+v=D_1}}(1+\phi(u))(1+\chi(-v)+\ol{\chi(-v)})=p+1+\sum_{\substack{u,v\in \F_{p} \\ u+v=D_1}}\phi(u)(\chi(-v)+\ol{\chi(-v)})\\
&=p+1+\phi(D_1)\chi(D_1)\sum_{\substack{u,v\in \F_{p} \\ u+v=1}}\phi(u)\chi(v)+\phi(D_1)\ol{\chi(D_1)}\ol{\sum_{\substack{u,v\in \F_{p} \\ u+v=1}}\phi(u)\chi(v)}\\
&=p+1+\chi(D_1)J(\phi,\chi)+\ol{\chi(D_1)}\ol{J(\phi,\chi)}.
\end{align*}
By Lemma~\ref{Ja} \eqref{pp}, we have $\#(\E_{\la}\otimes \F_{p})(\F_{p})=p+1-(\chi(D_1)\chi(4)\mu+\ol{\chi(D_1)\chi(4)}\ol{\mu})$. By noting that $\left(\fr{\cdot}{\mu}\right)_{3}:(\Z[\ro]/(\mu))^{\times}\simeq \F_{p}^{\times} \stackrel{\chi}{\rightarrow} W'$, 
and $\left(\fr{\la}{\mu}\right)_{3}=\left(\fr{\mu}{\la}\right)_{3}=\chi_{\la}(\mu)$ (Lemma~\ref{Ja} \eqref{rec}), we have $\chi(4D_1)=\left(\fr{\la}{\mu}\right)_{3}^2=\ol{\left(\fr{\la}{\mu}\right)_{3}}=\ol{\chi_{\la}(\mu)}$. Thus, $\#(\E_{\la}\otimes \F_{p})(\F_{p})=p+1-(\ol{\chi_{\la}(\mu)}\mu+\chi_{\la}(\mu)\ol{\mu})$ holds.

\eqref{oq}
For $q\neq 2$, $y^2=x^3+\fr{\la^2}{4}$ is the minimal Weierstrass equation at the prime ideal $(-q)$.
Let $\phi: \F_{q^2}^{\times}\rightarrow \{\pm1\}$ be the quadratic residue symbol, and $\chi: \F_{q^2}^{\times}\rightarrow W'$ be the cubic residue symbol. Let $D_2\in \F_{q^2}$ be the image of $\fr{\la^2}{4}$ in $\Z[\ro]/(-q)\simeq \F_{q^2}$. Note that $D_2\neq 0, \phi(D_2)=1, \chi(-1)=1$. Then, we have
\begin{align*}
\#(\E_{\la}\otimes \F_{q^2})(\F_{q^2})&=1+\#\{(x,y) \in  \F_{q^2}\times \F_{q^2} \mid y^2=x^3+D_2\}=1+\sum_{\substack{u,v\in \F_{q^2} \\ u+v=D_2}}\#\{y\in \F_{q^2}\mid y^2=u\}\cdot\#\{x\in \F_{q^2}\mid x^3=-v\}\\
&=1+\sum_{\substack{u,v\in \F_{q^2} \\ u+v=D_2}}(1+\phi(u))(1+\chi(-v)+\ol{\chi(-v)})=q^2+1+\sum_{\substack{u,v\in \F_{q^2} \\ u+v=D_2}}\phi(u)(\chi(-v)+\ol{\chi(-v)})\\
&=q^2+1+\phi(D_2)\chi(D_2)\sum_{\substack{u,v\in \F_{q^2} \\ u+v=1}}\phi(u)\chi(v)+\phi(D_2)\ol{\chi(D_2)}\ol{\sum_{\substack{u,v\in \F_{q^2} \\ u+v=1}}\phi(u)\chi(v)}\\
&=q^2+1+\chi(D_2)J_{2}(\phi,\chi)+\ol{\chi(D_2)}\ol{J_{2}(\phi,\chi)}.
\end{align*}
By Lemma~\ref{Ja} \eqref{qq}, we have $\#(\E_{\la}\otimes \F_{q^2})(\F_{q^2})=q^2+1+q(\chi(D_2)+\ol{\chi(D_2)})$.
By noting that $\left(\fr{\cdot}{-q}\right)_{3}:(\Z[\ro]/(-q))^{\times}\simeq \F_{q^2}^{\times} \stackrel{\chi}{\rightarrow} W'$, $\left(\fr{-2}{-q}\right)_{3}=\left(\fr{-q}{-2}\right)_{3}=1$, and $\left(\fr{\la}{-q}\right)_{3}=\left(\fr{-q}{\la}\right)_{3}=\chi_{\la}(-q)=\chi_{\la}(q)$ (Lemma~\ref{Ja} \eqref{rec}), we have $\chi(D_2)=\left(\fr{\la}{-q}\right)_{3}^2=\ol{\left(\fr{\la}{-q}\right)_{3}}=\ol{\chi_{\la}(q)}$. Thus, $\#(\E_{\la}\otimes \F_{q^2})(\F_{q^2})=q^2+1+q(\chi_{\la}(q)+\ol{\chi_{\la}(q)})$ holds.

Let $q=2$ and $D_3\in \F_{4}$ be the image of $\la$ in $\Z[\ro]/(-2)\simeq \F_{4}$.
Note that $D_3\neq 0$ and $y^2+\la y=x^3$ is the minimal Weierstrass equation at the prime ideal $(-2)$.
For $y\in \F_{4}$, we show that $y^2+D_3y=1$ if and only if $(D_3+1)y=0$, $y\neq 0,1$. Indeed, for $y\in \F_{4}$ with $ y^2+D_3y=1$, we have $y\neq 0, 1$, since $D_3\neq 0$.
Thus, $y^2=y+1$ holds, since $0=y^3-1=(y-1)(y^2+y+1)$ and $y\neq 1$.
Thus, we have $(D_3+1)y=0$ from $y^2+D_3y=1$. 
For $(D_3+1)y=0, y\neq 0,1$, since $y^2=y+1$, we have $y^2+D_3y=y+1+D_3y=1$.
Therefore, 
\begin{align*}
\#\{(x,y)\in  \F_{4}^{\times}\times \F_{4} \mid y^2+D_3y=1\}&=3\#\{y\in \F_{4} \mid y^2+D_3y=1\}=3\#\{y\in \F_{4} \mid (D_3+1)y=0, y\neq 0,1 \}\\
&=\left\{
\begin{array}{ll}
6 & \text{if}\ D_3=1,\\
0 & \text{if}\ D_3\neq 1
\end{array}
\right.
\end{align*}
holds.
Then, by noting that $\#\{y\in \F_{4} \mid y^2+D_3y=0\}=\#\{0,-D_3\}=2$, we have 
\begin{align*}
\#(\E_{\la}\otimes \F_{4})(\F_{4})&=1+\#\{(x,y)\in  \F_{4}\times \F_{4} \mid y^2+D_3y=x^3\}\\
&=1+\#\{(x,y)\in  \F_{4}^{\times}\times \F_{4} \mid y^2+D_3y=1\}+\#\{y\in \F_{4} \mid y^2+D_3y=0\}\\
&=3+\#\{(x,y)\in  \F_{4}^{\times}\times \F_{4} \mid y^2+D_3y=1\}=\left\{
\begin{array}{ll}
9 & \text{if}\ D_3=1,\\
3 & \text{if}\ D_3\neq 1.
\end{array}
\right.
\end{align*}
On the other hand, by Lemma~\ref{Ja} \eqref{rec}, we have
$$
\chi_{\la}(2)=\chi_{\la}(-2)=\left(\fr{\la}{-2}\right)_{3}=\left\{
\begin{array}{ll}
1 & \text{if}\ \la\equiv 1 \bmod 2,\\
\ro \ \text{or}\ \rob & \text{if}\ \la \equiv \ro \ \text{or}\ \rob \bmod 2.
\end{array}
\right.
$$
By noting that $\ro+\rob=-1$, 
$$
q^2+1+q(\chi_{\la}(q)+\ol{\chi_{\la}(q)})=5+2(\chi_{\la}(2)+\ol{\chi_{\la}(2)})=\left\{
\begin{array}{ll}
9 & \text{if}\ \la\equiv 1 \bmod 2,\\
3 & \text{if}\ \la \equiv \ro \ \text{or}\ \rob \bmod 2
\end{array}
\right.
$$
holds. Since $\la \equiv 1 \bmod 2 \Longleftrightarrow D_3=1$, we have $\#(\E_{\la}\otimes \F_{q^2})(\F_{q^2})=q^2+1+q(\chi_{\la}(q)+\ol{\chi_{\la}(q)})$ for $q=2$.

\eqref{rat}
If $\infty \neq (x,y)\in \E_{\la}(\Q(\ro))[2]$, then, $y=0$. Then, $x^3=-\fr{\la^2}{4}$ and $x\in \Q(\ro)$. This is a contradiction. Thus, we have $\E_{\la}(\Q(\ro))[2]=\{\infty\}$. 

Let $m\geq 1$ be an integer that is relatively prime to $5$. 
Then, $\E_{\la}(\Q(\ro))[m]$ has injective homomorphisms to $\Z/m\Z\times \Z/m\Z$ and $(\E_{\la}\otimes \F_{25})(\F_{25})$ by \cite[Chap. VII, Proposition3.1 (b)]{Sil'}.
On the other hand, $\#(\E_{\la}\otimes \F_{25})(\F_{25})=26+5(\chi_{\la}(5)+\ol{\chi_{\la}(5)})=36,$ or $21$ by \eqref{oq}.
Thus, for $m$ relatively prime to $2,3,5,7$, we have $\#\E_{\la}(\Q(\ro))[m]=1$.
We also have that $\E_{\la}(\Q(\ro))[m]$ has injective homomorphisms to $\Z/m\Z\times \Z/m\Z$ and $(\E_{\la}\otimes \F_{121})(\F_{121})$ for $m$ relatively prime to $11$ by loc. cit.
On the other hand, $\#(\E_{\la}\otimes \F_{121})(\F_{121})=122+11(\chi_{\la}(11)+\ol{\chi_{\la}(11)})=144,$ or $111$ by \eqref{oq}. 
Thus, for $m$ relatively prime to $2,3,11,37$, we have $\#\E_{\la}(\Q(\ro))[m]=1$.

Therefore, we have $\E_{\la}(\Q(\ro))_{\text{tors}}=\E_{\la}(\Q(\ro))[3]$.
Let $\infty \neq P=(x,y)\in \E_{\la}(\Q(\ro))[3]$. If $y=0$, we have $P\in \E_{\la}(\Q(\ro))[3]\cap \E_{\la}(\Q(\ro))[2]=\{\infty\}$. This is a contradiction. Thus, we have $y\neq 0$, equivalently, $4x^3+\la^2\neq 0$. For $Q\in \E_{\la}$, we denote  the $x$-coodinate of $Q$ by $x(Q)$. 
Then, $x(2P)=\fr{x^4-2\la^2x}{4x^3+\la^2}$, $x(-P)=x$.
Since $2P=-P$, we have $\fr{x^4-2\la^2x}{4x^3+\la^2}=x$. Thus, $x=0$ or $x^3=-\la^2$. Since $x \in \Q(\ro)$, we have $x^3\neq -\la^2$. Therefore, $x=0$. By $P\in \E_{\la}:y^2=x^3+\fr{\la^2}{4}$, we have $y=\pm \fr{\la}{2}$.
Thus, we proved $\E_{\la}(\Q(\ro))_{\text{tors}}=\E_{\la}(\Q(\ro))[3]=\{\infty, (0, \pm \fr{\la}{2})\}$.
\end{proof}

\begin{proposition} (Correspondence between the Hecke $L$-series and the $L$-function of the elliptic curve)\label{Deu}
Let $L(\E_{\la}/\Q(\ro),s)$ be the $L$-function of the elliptic curve $\E_{\la}:y^2=x^3+\fr{\la^2}{4}$ over $\Q(\ro)$ (see \cite[\S 10 of Chap. II]{Sil}).
The Hecke $L$-series defined by \eqref{defl} corresponds to the $L$-function of the elliptic curve $\E_{\la}$, i.e., one has
$$
L(s,\widetilde{\chi_{\la}})L(s,\ol{\widetilde{\chi_{\la}}})=L(\E_{\la}/\Q(\ro),s).
$$
\end{proposition}
\begin{proof}
We recall the definition of $L(\E_{\la}/\Q(\ro),s)$: 
$L(\E_{\la}/\Q(\ro),s):=\prod_{\p: \text{prime ideal}}L_{\p}(\E_{\la}/\Q(\ro),q_{\p}^{-s})^{-1}$, where $L_{\p}(\E_{\la}/\Q(\ro),T)$ is the local $L$-factor of $\E_{\la}/\Q(\ro)$ at $\p$ defined as follows: Let
$q_{\p}:=\#k_{\p}$, $a_{\p}:=q_{\p}+1-\#(\E_{\la}\otimes k_{\p})(k_{\p})$, if $\E_{\la}$ has good reduction at $\p$. Then, 
$$
L_{\p}(\E_{\la}/\Q(\ro),T):=1-a_{\p}T+q_{\p}T^2.
$$
If $\E_{\la}$ has bad reduction at $\p$, we define
$$
L_{\p}(\E_{\la}/\Q(\ro),T):=\left\{
\begin{array}{ll}
1-T &  \text{if $\E_{\la}$ has split multiplicative reduction at $\p$,}\\
1+T &  \text{if $\E_{\la}$ has non-split multiplicative reduction at $\p$,}\\
1&  \text{if $\E_{\la}$ has additive reduction at $\p$.}
\end{array}
\right.
$$
If $\p=(\mu)$, where $\mu \in \Z[\ro]$ is a primary prime $\neq \la$ such that $p=\mu\ol{\mu}$ is a rational prime, then
$k_{\p}\simeq \F_{p},\ q_{\p}=p,\ a_{\p}=\chi_{\la}(\mu)\ol{\mu}+\ol{\chi_{\la}}(\mu)\mu$ hold by Lemma~\ref{E} \eqref{op}. Thus, 
$$
L_{\p}(\E_{\la}/\Q(\ro),q_{\p}^{-s})=1-(\chi_{\la}(\mu)\ol{\mu}+\ol{\chi_{\la}}(\mu)\mu)p^{-s}+p^{1-2s}.
$$
On the other hand, the local $L$-factor of $L(s,\widetilde{\chi_{\la}})L(s,\ol{\widetilde{\chi_{\la}}})$ at $\p$ is 
$$
(1-\chi_{\la}(\mu)\ol{\mu}(\mu\ol{\mu})^{-s})(1-\ol{\chi_{\la}}(\mu)\mu(\mu\ol{\mu})^{-s})
=1-(\chi_{\la}(\mu)\ol{\mu}+\ol{\chi_{\la}}(\mu)\mu)p^{-s}+p^{1-2s}
$$
by \eqref{defl}.
Therefore, the local $L$-factor of $L(s,\widetilde{\chi_{\la}})L(s,\ol{\widetilde{\chi_{\la}}})$ at $\p$ coinsides with the one of $L(\E_{\la}/\Q(\ro),s)$ at $\p$.

Next, if $\p=(-q)$, where $q$ is a rational prime such that $q\equiv 2\bmod 3$, then $k_{\p}\simeq \F_{q^2},\ q_{\p}=q^2,\ a_{\p}=-q(\chi_{\la}(q)+\ol{\chi_{\la}(q)})$ hold by Lemma~\ref{E} \eqref{oq}. Thus,  
$$
L_{\p}(\E_{\la}/\Q(\ro),q_{\p}^{-s})=1+(\chi_{\la}(q)+\ol{\chi_{\la}}(q))q^{1-2s}+q^{2-4s}.
$$
On the other hand, the local $L$-factor of $L(s,\widetilde{\chi_{\la}})L(s,\ol{\widetilde{\chi_{\la}}})$ at $\p$ is 
$$
(1-\chi_{\la}(-q)(\ol{-q})q^{-2s})(1-\ol{\chi_{\la}}(-q)(-q) q^{-2s})=1+(\chi_{\la}(q)+\ol{\chi_{\la}}(q))q^{1-2s}+q^{2-4s}
$$
by \eqref{defl}.
Therefore, the local $L$-factors of $L(s,\widetilde{\chi_{\la}})L(s, \ol{\widetilde{\chi_{\la}}})$ at $\p$ coinsides with the one of $L(\E_{\la}/\Q(\ro),s)$ at $\p$.

If $\p=(\la), (1-\ro)$, then, $\E_{\la}$ has additive reduction at $\p$ by Lemma~\ref{E} \eqref{red}. 
Then, both of the local $L$-factors of $L(s,\widetilde{\chi_{\la}})L(s, \ol{\widetilde{\chi_{\la}}})$ and $L(\E_{\la}/\Q(\ro),s)$ are equal to 1.
\end{proof}

\subsection{Elliptic Gauss sums and Tate-Shafarevich groups}
In this section, we show that the order of the Tate-Shafarevich group of the elliptic curve $\E_{\la}/\Q(\ro)$ is equal to the square of the absolute value of the coefficient $\al$ of the relevant elliptic Gauss sum or its quarter under the BSD conjecture on the $2, 3$-parts of the leading term for $\E_{\la}/\Q(\ro)$.
By combining this and Theorem~\ref{main1}, we obtain that the order of the Tate-Shafarevich group of $\E_{\la}/\Q(\ro)$ is congruent to the square of Bernoulli-Hurwitz-type numbers $C_{n}, D_{n}$ times $1/9$ or $1/36$ modulo $\ell$ under the BSD conjecture on the $2, 3$-parts of the leading term.
Unconditionally, we have such congruences up to the multiplications of powers of $2, 3$.

\begin{lemma} (Hecke $L$-value and elliptic Gauss sum)\label{L1}
\begin{equation*}
L(1,\widetilde{\chi_{\la}})=\left\{
\begin{array}{ll}
\displaystyle -\fr{\pe\chi_{\la}(3)G_{\la}(\chi_{\la},\varphi)}{\la}=-\fr{\pe\chi_{\la}(3)}{\widetilde{\la}}\al & (\ell \equiv 7 \bmod 9),\\
\displaystyle -\fr{\pe\overline{\chi_{\la}}(3)G_{\la}(\chi_{\la},\varphi^{-1})}{\la}=-\fr{\pe\overline{\chi_{\la}}(3)}{\widetilde{\la}}\al & (\ell \equiv 4 \bmod 9).
\end{array}
\right.
\end{equation*}
Moreover, we have $L(1, \widetilde{\chi_{\la}})\neq 0$ in these cases.
\end{lemma}
\begin{proof}
The lemma follows from \cite[Theorem 1.16 of \S 1.4]{Asai} and Theorem~\ref{defco}.
\end{proof}

\begin{theorem} (The Tate-Shafarevich groups and the coefficients of elliptic Gauss sums)\label{bsd}
For the number $\al$ defined in Theorem~\ref{defco}, one has
\begin{equation*}
\#\Sha(\E_{\la}/\Q(\ro))=
\left\{
\begin{array}{ll}
|\al|^2/4 & \text{if\ $\ell \equiv 4 \bmod 9$ and $n \equiv 0\bmod 3$}\  (\text{where\ }\la=3m+1+3n\ro,\ m,n\in \Z), \\
|\al|^2 &  \text{otherwise,}
\end{array}
\right.
\end{equation*}
up to the multiplications of powers of $2,3$. Moreover, for $p=2,3$, under the BSD conjecture on the $p$-part of the leading term for the elliptic curve $\E_{\la}$, the above equality holds at powers of $p$ as well.
\end{theorem}
\begin{proof}
Let $\tau_{\infty}:=\varpi_{\la}\ol{\varpi_{\la}}$, where $\varpi_{\la}$ is a generator, as $\Z[\ro]$-modules, over $\Z[\ro]$ of the period lattice, in $\C$, of $\E_{\la}$ (Note that $\varpi_{\la}$ is not unique, however, $\tau_{\infty}$ is well-defined).
The substitution $y\mapsto \fr{y}{2}$ gives an equation of the form 
$$
\E_{\la}: y^2=4x^3+\la^2.
$$
From this equation, we have $g_{\la}:=\sum_{0\neq \omega\in \varpi_{\la}\Z[\ro]}\fr{1}{\omega^{6}}=-\fr{\la^2}{140}$ by \cite[Theorem 3.5 (b) of Chap. VI]{Sil'}.
On the other hand, by Lemma~\ref{ee} \eqref{weq}, $g_{1}:=\sum_{0\neq \omega\in \pe\Z[\ro]}\fr{1}{\omega^{6}}=\fr{27}{140}$ by loc. cit.
Thus, we have $\fr{\varpi_{\la}^6}{\pe^6}=\fr{g_{1}}{g_{\la}}=-\fr{27}{\la^2}.$
By taking the absolute values, $\tau_{\infty}^3=\fr{3^3}{\ell}\pe^6$ since $\ell=\la\ol{\la}=|\la|^2$.
Then, we have 
$$
\tau_{\infty}=\fr{3\pe^2}{\ell^{\fr{1}{3}}}.
$$
We recall that 
$$
\tau_{(\la)}=3, \ \#\E_{\la}(\Q(\ro))_{\text{tors}}=3
$$
by Lemma~\ref{E} \eqref{type1}, \eqref{rat}, respectively. 
If $s\in \R$, $\ol{L(s,\widetilde{\chi_{\la}})}=L(s,\ol{\widetilde{\chi_{\la}}})$. Thus, $L(1,\widetilde{\chi_{\la}})\ol{L(1,\widetilde{\chi_{\la}})}=|L(1,\widetilde{\chi_{\la}})|^2$.
Therefore, 
$$
L(\E_{\la}/\Q(\ro),1)=L(1,\widetilde{\chi_{\la}})L(1,\ol{\widetilde{\chi_{\la}}})=\fr{|\al|^2}{\ell^{\fr{1}{3}}}\pe^2
$$
by Proposition~\ref{Deu} and Lemma~\ref{L1} (note that $\widetilde{\la}^3=\la$, $(|\widetilde{\la}|^2)^3=|\la|^2=\ell$).
Thus, $L(\E_{\la}/\Q(\ro),1)\neq 0$ holds since  Lemma~\ref{L1}.
Hence, by \cite[Theorem (i)]{R}, we have $\#\Sha(\E_{\la}/\Q(\ro))<\infty$ and 
\begin{equation} \label{B1}
L(\E_{\la}/\Q(\ro),1)=\tau_{\infty}\tau_{(1-\ro)}\tau_{(\la)}\fr{\#\Sha(\E_{\la}/\Q(\ro))}{\#\E_{\la}(\Q(\ro))_{\text{tors}}^2},
\end{equation}
up to the multiplications of powers of $2,3$. 
If $\ell \equiv 4 \bmod 9$ and $n \equiv 0\bmod 3$, then $\tau_{(1-\ro)}=4$  by Lemma~\ref{E} \eqref{type1}. Thus,  
$$
\fr{|\al|^2}{\ell^{\fr{1}{3}}}\pe^2=\fr{3\pe^2}{\ell^{\fr{1}{3}}}\cdot 4 \cdot 3 \cdot \fr{\#\Sha(\E_{\la}/\Q(\ro))}{3^2},
$$
up to the multiplications of powers of $2, 3$. Therefore, we have $\#\Sha(\E_{\la}/\Q(\ro))=\fr{|\al|^2}{4}$ up to the multiplications of powers of $2,3$.

Otherwise, $\tau_{(1-\ro)}=1$, we have
$$
\fr{|\al|^2}{\ell^{\fr{1}{3}}}\pe^2=\fr{3\pe^2}{\ell^{\fr{1}{3}}}\cdot 1 \cdot 3 \cdot \fr{\#\Sha(\E_{\la}/\Q(\ro))}{3^2},
$$
up to the multiplications of powers of $2,3$. Therefore, we have $\#\Sha(\E_{\la}/\Q(\ro))=|\al|^2$ up to the multiplications of powers of $2,3$. 

Finally, for $p=2,3$, under the BSD conjecture on the $p$-part of the leading term for the elliptic curve $\E_{\la}$, the equality \eqref{B1} holds at powers of $p$ as well. Hence the final claim of the theorem follows. 
\end{proof}

\begin{remark}By the second supplementary law of the cubic reciprocity law, $n\equiv 0\bmod 3 \Longleftrightarrow \chi_{\la}(3)=1$ holds, since $\chi_{\la}(3)=\ro^n$ by \cite[the last formula in Theorem 7.8 of Chap. 7]{Le}, where $n$ is as in Theorem~\ref{bsd}. (Note that the definition of primarity in \cite[\S 7.1]{Le} is slightly different from the one in this paper, however, it does not effect the above equivalence.)
Hence, if $\ell\equiv 7,4 \bmod 9$ and $n\equiv 0\bmod 3$, we have $\al\in \Z$ by Theorem~\ref{defco}.
Under the BSD conjecture on the $2$-part of the leading term for the elliptic curve $\E_{\la}$, we have
\begin{equation} \label{cje}
n\equiv 0\bmod 3 \Longrightarrow 2|\al
\end{equation}
since Theorem~\ref{bsd}.
As long as the center of the table at \cite[Table 1. in p.117]{Asai}, \eqref{cje} is plausible. However, by the case of $p=139$ in Asai's table, we cannot replace $``\Longrightarrow"$ with $``\Longleftrightarrow"$ in \eqref{cje}.
\end{remark}

\begin{corollary} (The Tate-Shafarevich groups and the Bernoulli-Hurwitz-type numbers)\label{mainco}
We write $\la=3m+1+3n\ro,\ m,n\in \Z$, then we have
\begin{equation} \label{B2}
\#\Sha(\E_{\la}/\Q(\ro))\underset{\bmod \la}{\equiv}
\left\{
\begin{array}{ll}
\left(\chi_{\la}(3)D_{\fr{2(\ell-1)}{3}}\right)^2/36 & \text{if\ $\ell \equiv 4 \bmod 9$ and $n \equiv 0\bmod 3$}, \\
\left(\chi_{\la}(3)D_{\fr{2(\ell-1)}{3}}\right)^2/9 & \text{if\ $\ell \equiv 4 \bmod 9$ and $n \equiv \pm1 \bmod 3$}, \\
\left(\ol{\chi_{\la}(3)}C_{\fr{2(\ell-1)}{3}}\right)^2/9 &  \text{if\ $\ell \equiv 7 \bmod 9$,}
\end{array}
\right.
\end{equation}

$$
\#\Sha(\E_{\la}/\Q(\ro))\underset{\bmod \ell}{\equiv}
\left\{
\begin{array}{ll}
\left(3^{\fr{\ell-4}{3}}D_{\fr{2(\ell-1)}{3}}\right)^2/4 & \text{if\ $\ell \equiv 4 \bmod 9$ and $n \equiv 0\bmod 3$}, \\
\left(3^{\fr{\ell-4}{3}}D_{\fr{2(\ell-1)}{3}}\right)^2 & \text{if\ $\ell \equiv 4 \bmod 9$ and $n \equiv \pm1 \bmod 3$}, \\
\left(3^{\fr{\ell-7}{6}}C_{\fr{2(\ell-1)}{3}}\right)^2 &  \text{if\ $\ell \equiv 7 \bmod 9$,}
\end{array}
\right.
$$
up to the multiplications of powers of $2, 3$. Moreover, for $p=2,3$, under the BSD conjecture on the $p$-part of the leading term for the elliptic curve $\E_{\la}$, the congruence \eqref{B2} holds without the ambiguity of the multiplications of powers of $p$.
\end{corollary}
\begin{proof}
These congruences hold from Theorem~\ref{main1}, Theorem~\ref{bsd} and Lemma~\ref{ab}.
\end{proof}
\begin{remark}
There exist infinitely many rational primes $\ell$ such that $\ell\equiv 7,4\bmod 9$ and $\F_{\ell}^{\times}/\langle2,3 \bmod \ell\rangle$ is not trivial, by Chebotarev's density theorem, where $\langle2,3 \bmod \ell\rangle$ is the subgroup generated by $2,3 \bmod \ell$ of $\F_{\ell}^{\times}$.
\end{remark}


\appendix
\section{Appendix}
In this chapter, we show some lemmas and formulae which are very interesting, although they are not used in the proofs of the main results in this paper. In \S A.1, we show lemmas on rational primes dividing the denominator of $d_{n}$.
In \S A.2, we determine the signs and the non-vanishing of $C_{n}$, $D_{n}$ (excluding $D_{6n+2}$), and the Eisenstein series $G_{n}$ associated to the lattice $\pe\Z[\ro]$ (see Lemma~\ref{DG}) by using new recurrence relations of them arising from new formulae of $\Sl$.
In \S A.3, we give numerical tables of $c_{n}, d_{n},$ $G_{n}$, and $BH_{n}$.

\subsection{On the denominator of $d_{n}$}
First, we show a recurrence relation including $G_{6n}$ and $D_{n}$.
\begin{lemma} (Relations between $G_{n}$ and $D_{n}$) \label{DG}
Let 
$$
G_{n}:=\sum_{0\neq \omega\in \pe\Z[\ro]}\fr{1}{\omega^{n}}
$$
for $n\geq 1$ (Eisenstein series associated to the lattice $\pe\Z[\ro]$). Then, for $m\geq 1$, we have
\begin{equation} \label{D}
\sum_{k=0}^mD_{6k+2}(6m-6k-1)G_{6(m-k)}3^{6k}=0,\ \ D_{6m-1}=\fr{(-1)^m3^{3m-1}-1}{2\cdot3^{6m-1}}G_{6m}.
\end{equation}
The second equality of \eqref{D} holds for $m=0$ as well, where $G_{0}:=-1$.
\end{lemma}
\begin{proof}
First, we note that $G_{6n}\in \Q$ for $n\geq 0$, since $\wp(u)=\fr{1}{u^2}+\sum_{n=1}^{\infty}(6n-1)G_{6n}u^{6n-2}$ around $u=0$ (note that $G_{n}=0$ unless $n$ is congruent to $0$ modulo $6$) and $\wp''=6\wp^2$. We prove the first equality of \eqref{D}. Since Lemma~\ref{pro} \eqref{g}, we have
$$
\left(\fr{1}{\Sl(-3u)}+\fr{1}{\Sl(3u)}\right)\x{u}
=3.
$$
Around $u=0$, we have $\wp(u)=\sum_{n=0}^{\infty}(6n-1)G_{6n}u^{6n-2}$, and $\fr{1}{\Sl(-3u)}+\fr{1}{\Sl(3u)}=18\sum_{n=0}^{\infty}D_{6n+2}3^{6n}u^{6n+2}$ by the definitions. Thus, 
$18\sum_{n=0}^{\infty}D_{6n+2}3^{6n}u^{6n}\sum_{n=0}^{\infty}(6n-1)G_{6n}u^{6n}=3$ follows. By comparing the coefficients of $u^{6m}$ of this equality for $m\geq 1$, the first equality of \eqref{D} follows.

We prove the second equality of \eqref{D}. Since Lemma~\ref{pro} \eqref{o} and $3=(1-\ro)(1-\rob)$, we have 
$$
3\left(\fr{1}{\Sl(-3u)}-\fr{1}{\Sl(3u)}\right)=(1-\ro)\zt{(1-\ro)u}-3\zt{u}.
$$
Around $u=0$, $\zt{u}=-\sum_{n=0}^{\infty}G_{6n}u^{6n-1}$, and $3\left(\fr{1}{\Sl(-3u)}-\fr{1}{\Sl(3u)}\right)=-2\sum_{n=0}^{\infty}D_{6n-1}3^{6n}u^{6n-1}$ by the definitions. Thus,
$$
-2\sum_{n=0}^{\infty}D_{6n-1}3^{6n}u^{6n-1}=-\sum_{n=0}^{\infty}G_{6n}\{(1-\ro)^{6n}-3\}u^{6n-1}=-\sum_{n=0}^{\infty}G_{6n}\{(-1)^{n}3^{3n}-3\}u^{6n-1}
$$
follows, since $(1-\ro)^{6n}=(-1)^n3^{3n}$. By comparing the coefficients of $u^{6n-1}$ of this equality for $n\geq 0$, the second equality of \eqref{D} is proved.
\end{proof}


\begin{definition} (The Bernoulli-Hurwitz number for the period lattice $\pe\Z[\ro]$)
Set $BH_{6m}:=(6m)!G_{6m}$ for $m\geq 1$. (Note that $BH_{6m}$ is the Bernoulli-Hurwitz number for the elliptic curve $y^2=4x^3-27$ defined in \cite{Katz}.)
\end{definition}

\begin{lemma} (The denominator of $D_{6m+2}$) \label{denX}
For a rational prime $\ell \equiv 1\bmod 3$ and $m \geq 0$, we define
$$
t(\ell,m):=\left\lfloor\fr{6m}{\ell-1}\right\rfloor,
$$
where $\lfloor\cdot\rfloor$ denotes the floor function.
%
Then, 
$$
\prod_{\substack{\ell: \text{rational prime} \\ \ell \equiv 1 \bmod 3}}\ell^{t(\ell, m)}\cdot 2(6m)!D_{6m+2}\in \Z\left[\fr{1}{3}\right].
$$
Note that the product on the left hand side is a finite product.
\end{lemma}
\begin{proof}
We write $h_{m}:=\prod_{\substack{\ell: \text{rational prime} \\ \ell \equiv 1 \bmod 3}}\ell^{t(\ell, m)}\cdot 2(6m)!D_{6m+2}$ for $m \geq 0$. Since $h_{0}=2D_{2}=\fr{1}{3}$ (see Table~\ref{td}), the lemma holds for $m=0$.
For $m\geq 1$, we assume that $h_{k}\in \Z\left[\fr{1}{3}\right]$ for $k=0,\ldots, m-1$.
By the first equality of \eqref{D}, we have
\begin{equation} \label{KK}
h_{m}3^{6m}=-\sum_{k=0}^{m-1}h_{k}(6m-6k-1)\left(\prod_{\substack{\ell: \text{rational prime} \\ \ell \equiv 1 \bmod 3}}\ell^{t(\ell, m)-t(\ell,k)}\right)BH_{6(m-k)}\binom{6m}{6k}3^{6k}.
\end{equation}
Since the elliptic curve $y^2=4x^3-27$ has good reduction (resp. good ordinary reduction) at $\ell$ if and only if $\ell\neq 3$ (cf. the proof of Lemma~\ref{E} \eqref{red}) (resp. $\ell\equiv 1\bmod 3$ (see \cite[Chap II, exer. 2.30]{Sil})),
it holds that $\left(\prod_{\substack{\ell \ \text{is rational prime}\\ \ell-1|6m \\ \ell\equiv 1\bmod 3}}\ell\right) BH_{6m}\in \Z[\fr{1}{3}]$ by \cite[Theorem]{Katz}. In terms of $m$, $t(\ell,m)$ is nondecreasing.
For $k=0,\ldots ,m-1$ and a rational prime $\ell\equiv 1 \bmod 3$, if $\ell-1$ divides $6(m-k)$, we have $t(\ell,m)>t(\ell,k)$, since $t(\ell,m)=\left\lfloor\fr{6m}{\ell-1}\right\rfloor=\fr{6(m-k)}{\ell-1}+\left\lfloor\fr{6k}{\ell-1}\right\rfloor>t(\ell,k)$.
Therefore, the right hand side of \eqref{KK} is $\in \Z \left[\fr{1}{3}\right]$.
\end{proof}

\begin{remark}
Numerically, we guess that 
$$
BH_{6m}\in \Z_{(3)},
$$
where $\Z_{(3)}$ is the localization of $\Z$ by the prime ideal $3\Z$ (see Table~\ref{tg}). If this is true, then the same proof of lemma also shows that 
$$
2\cdot 3^{6m+1}(6m)!L(m)\cdot D_{6m+2}\in \Z,
$$
where let $L(m):=\prod_{\substack{\ell: \text{rational prime} \\ \ell \equiv 1 \bmod 3}}\ell^{t(\ell, m)}$.
\end{remark}

\begin{lemma} (Integrality of $d_{6m+2}$ outside $3$) \label{dend'}
For $m\geq 0$, we have $d_{6m+2}\in \Z\left[\fr{1}{3}\right]$.
\end{lemma}
\begin{proof}
By Lemma~\ref{dend}, we can write $2d_{6m+2}=\fr{N_{m}}{3^{r(m)}}$, $r(m)\geq 0$, $N_{m}\in \Z$.
By Lemma~\ref{denX}, we have $2L(m)\cdot \fr{d_{6m+2}}{(6m+1)(6m+2)}\in \Z\left[\fr{1}{3}\right]$.
Thus, $L(m)\cdot \fr{N_{m}}{(6m+1)(6m+2)}\cdot \fr{1}{3^{r(m)}}\in \Z\left[\fr{1}{3}\right]$ holds.
Then, we have $(6m+1)(6m+2)| L(m)\cdot N_{m}$, since $6m+1$ and $6m+2$ are relatively prime to $3$.
Especially, we have $2| N_{m}$, since $2$ does not divide $L(m)$.
Therefore, the lemma is proved.
\end{proof}

\begin{lemma}
\begin{enumerate}[(1)]
\item \label{fr3}
$$
\fr{1}{\Sl(\sqrt{-3}u)}+\fr{1}{\Sl(-\sqrt{-3}u)}=-\fr{\Sl(u)^2}{\Cl(u)}=\Sl(-u)\Sl(u) \in \mathcal{H}_{\Z}.
$$
\item (The denominator of $d_{6m+2}$)\label{d3}
$$
d_{6n+2}3^{3n+1}\in \Z.
$$

\end{enumerate}

\end{lemma}
\begin{proof}
\eqref{fr3}
We have $$
\fr{1}{\Sl(\sqrt{-3}u)}=\fr{1+\rob \Sl(u)^3}{\sqrt{-3}\Sl(u)\Cl(u)}
$$
by \cite[2. (viii) of Appendix]{Asai} or putting $u:=\ro u$ and $v:=-\rob u$ of the addition formula of $\Sl$ (Lemma~\ref{addsl'}).
By $\Sl(-u)=-\fr{\Sl(u)}{\Cl(u)}$ and $\Cl(-u)=\fr{1}{\Cl(u)}$ (Lemma~\ref{pro} \eqref{-}), we have
$$
\fr{1}{\Sl(-\sqrt{-3}u)}=\fr{-1-\ro \Sl(u)^3}{\sqrt{-3}\Sl(u)\Cl(u)}.
$$
Then, by two equalities and noting that $\ro-\rob=\sqrt{-3}$, the first equality of \eqref{fr3} is proved. The second equality of \eqref{fr3} also follows from $\Sl(-u)=-\fr{\Sl(u)}{\Cl(u)}$ (Lemma~\ref{pro} \eqref{-}).
By lemma~\ref{denc}, $\Sl(u), \Sl(-u)\in \mathcal{H}_{\Z}$. Thus, $\Sl(u)\Sl(-u)\in \mathcal{H}_{\Z}$. 

\eqref{d3} By the definition, we have $\fr{1}{\Sl(\sqrt{-3}u)}+\fr{1}{\Sl(-\sqrt{-3}u)}
=2\sum_{n=0}^{\infty}d_{6n+2}(-3)^{3n+1}\fr{u^{6n+2}}{(6n+2)!}$.
By \eqref{fr3}, we have $2d_{6n+2}(-3)^{3n+1}\in \Z$. By Lemma~\ref{dend'}, $d_{6n+2}(-3)^{3n+1}\in \Z$.

\end{proof}

\begin{lemma} (The denominator of $d_{6m-1}$) \label{den-1} 
For $m\geq 1$, we have
$$
2m\left(\prod_{\substack{\ell \ \text{is rational prime}\\ \ell-1|6m \\ \ell\equiv 1\bmod 3}}\ell\right) d_{6m-1}\in \Z\left[\fr{1}{3}\right].
$$
\end{lemma}
\begin{proof}
From the second equality of \eqref{D}, for $m\geq 1$, we have
$$
d_{6m-1}=\fr{(-1)^m3^{3m-1}-1}{2\cdot3^{6m-1}}\fr{BH_{6m}}{6m}=\fr{(-1)^m3^{3m-1}-1}{2}\cdot \fr{BH_{6m}}{2m} \cdot \fr{1}{3^{6m}}.
$$
Since $(-1)^m3^{3m-1}-1\equiv 2\bmod 4$, $\fr{(-1)^m3^{3m-1}-1}{2}\in \Z$.
From \cite[Theorem]{Katz}, $\left(\prod_{\substack{\ell \ \text{is rational prime}\\ \ell-1|6m \\ \ell\equiv 1\bmod 3}}\ell\right) BH_{6m}\in \Z[\fr{1}{3}]$. Then, $2m\left(\prod_{\substack{\ell \ \text{is rational prime}\\ \ell-1|6m \\ \ell\equiv 1\bmod 3}}\ell\right) d_{6m-1}\in \Z[\fr{1}{3}]$ holds.
\end{proof}

\subsection{On the signs and the non-vanishing of $C_{n}, D_{n},$ and $G_{n}$}
In this section, we determine the signs of $C_{n}, D_{n}$ (excluding $D_{6n+2}$), and $G_{n}$. Moreover, it is proved that $C_{3n+1}$, $D_{6n-1}$, $G_{6n}$, and $BH_{6n}$ do not vanish.
First, we determine the signs and the non-vanishing of $G_{n}$, $BH_{6n}$ and $D_{6n-1}$.
\begin{lemma}[The signs and the non-vanishing of $G_{6n}$, $BH_{6n}$, and $D_{6n-1}$] \label{sGD}
For $n\geq 2$, we have
\begin{equation} \label{rG}
(6n-1)\{(6n-2)(6n-3)-12\}G_{6n}=6\sum_{k=1}^{n-1}(6k-1)(6n-6k-1)G_{6k}G_{6(n-k)},
\end{equation}
and for $n\geq 1$, $G_{6n}>0$, hence $BH_{6n}>0$. Moreover, $(-1)^mD_{6m-1}>0$ holds for $m\geq 1$. In particular, $G_{6n}$, $BH_{6n}$, and $D_{6n-1}$ are non-zero.
\end{lemma}
\begin{proof}
Since $\wp''=6\wp^2$ and $\x{u}=\sum_{n=0}^\infty(6n-1)G_{6n}u^{6n-2}$, for $n\geq 1$, we have 
$$
(6n-1)(6n-2)(6n-3)G_{6n}=6\sum_{\substack{m+k=n\\ k,m\geq 0}}G_{6k}(6k-1)G_{6m}(6m-1)=6\sum_{k=1}^{n-1}(6k-1)(6n-6k-1)G_{6k}G_{6(n-k)}+12G_{6n}(6n-1),
$$
here, note that $G_{0}=-1$. Then, we have the equality \eqref{rG}.
By the proof of Lemma~\ref{ee} \eqref{weq}, we have $G_{6}=\fr{27}{140}>0$.
The positivity of $G_{6n}$ for $n\geq 1$ follows by induction on $n$, \eqref{rG} and $G_{6}>0$. 
Then, by the second equality of \eqref{D}, we have
$$
(-1)^mD_{6m-1}=(-1)^m\fr{(-1)^m3^{3m-1}-1}{2\cdot3^{6m-1}}G_{6m}=\fr{3^{3m-1}-(-1)^m}{2\cdot3^{6m-1}}G_{6m}>0
$$
for $m\geq 1$.
\end{proof}
\begin{definition}[Hurwitz coefficient]
For a formal power series $F(u)=\sum_{n=0}^{\infty}\fr{a_{n}}{n!}u^n \in \Q[[u]]$ in terms of $u$, we call $a_{n}$ the {\it Hurwitz coefficient} for $F(u)$ of $u^n$.
\end{definition}

\begin{lemma} 
\begin{enumerate}[(1)] 
\item \label{rc}
We have
$$
2\cdot 3^{6n}c_{6n+1}=2BH_{6n}(6n+1)(6n-1)-3^{6n}(6n+1)2n(6n-1)c_{6n-2}+\sum_{k=1}^n
c_{6(n-k)+1}BH_{6k}(6k-1)(6k-2)3^{6(n-k)}\dbinom{6n+1}{6k}
$$
for $n\geq 1$.

$$
2\cdot 3^{6n}c_{6n+4}=-(6n+4)(2n+1)(6n+2)3^{6n}c_{6n+1}+\sum_{k=1}^n c_{6(n-k)+4}BH_{6k}(6k-1)(6k-2)3^{6(n-k)}\dbinom{6n+4}{6k}
$$
for $n \geq 1$.
\item (The sign and the non-vanishing of $c_{n}$)\label{sc} We have $(-1)^{k+1}c_{k}>0$ for $k\geq 1$, i.e., $c_{6n+1}>0,\ c_{6n+4}<0$ for $n\geq 0$. In other words, all the Hurwitz coefficients for $-\Sl(-u)$ are $>0$. In particular, $c_{3n+1}$ does not vanish.
\item (The sign and the non-vanishing of $e_{n}$)\label{se}
We put $\Cl(u)=:\sum_{n=0}^{\infty}\fr{e_{3n}}{(3n)!}u^{3n}$, then $(-1)^{n}e_{3n}>0$ for $n\geq 0$. Equivalently, $e_{6n}>0,\ e_{6n+3}<0$ for $n\geq 0$. In particular, $e_{3n}$ does not vanish.

\end{enumerate}
\end{lemma}

\begin{proof}
\eqref{rc}
By $\Sl(u)=\vp{\fr{u}{-3\pe}}=\fr{6\x{\fr{u}{3}}}{9-\y{\fr{u}{3}}}$ (Lemma~\ref{pro} \eqref{phi}),
$\Sl(u)=\sum_{n=0}^\infty C_{3n+1}u^{3n+1}=\sum_{n=0}^\infty C_{6n+1}u^{6n+1}+\sum_{n=0}^\infty C_{6n+4}u^{6n+4}$,
$\x{u}=\sum_{n=0}^{\infty}G_{6n}(6n-1)u^{6n-2}$
and $\y{u}=\sum_{n=0}^{\infty}G_{6n}(6n-1)(6n-2)u^{6n-3}$, we have 
\begin{equation} \label{recc}
\left(\sum_{n=0}^\infty C_{6n+1}u^{6n}+\sum_{n=0}^\infty C_{6n+4}u^{6n+3}\right)\left(u^3-\sum_{n=0}^\infty G_{6n}(6n-1)(6n-2)\fr{u^{6n}}{3^{6n-1}}\right)=6\sum_{n=0}^\infty G_{6n}(6n-1)\fr{u^{6n}}{3^{6n}}.
\end{equation}
Taking the even part of \eqref{recc},
$$
\sum_{n=0}^\infty C_{6n+1}u^{6n}\left(-\sum_{n=0}^\infty G_{6n}(6n-1)(6n-2)\fr{u^{6n}}{3^{6n-1}}\right)+\sum_{n=0}^\infty C_{6n+4}u^{6n+3}\cdot u^3=6\sum_{n=0}^\infty G_{6n}(6n-1)\fr{u^{6n}}{3^{6n}}
$$
holds. By comparing the coefficients of $u^{6n}$, we have 
$$
3^{6n-1}C_{6n-2}=2G_{6n}(6n-1)-2\cdot 3^{6n}C_{6n+1}+\sum_{k=1}^n
C_{6(n-k)+1}G_{6k}(6k-1)(6k-2)3^{6(n-k)}
$$
for $n\geq 1$. Since $C_{k}=\fr{c_{k}}{k!},\ G_{6k}=\fr{BH_{6k}}{(6k)!}$, we also have
$$
3^{6n}(6n+1)2n(6n-1)c_{6n-2}=2BH_{6n}(6n+1)(6n-1)-2\cdot 3^{6n}c_{6n+1}+\sum_{k=1}^n
c_{6(n-k)+1}BH_{6k}(6k-1)(6k-2)3^{6(n-k)}\dbinom{6n+1}{6k}
$$
for $n\geq 1$. This is the first equality of \eqref{rc}.
Taking the odd part of \eqref{recc}, 
$$
\sum_{n=0}^\infty C_{6n+1}u^{6n}\cdot u^3+\sum_{n=0}^\infty C_{6n+4}u^{6n+3}\left(-\sum_{n=0}^\infty G_{6n}(6n-1)(6n-2)\fr{u^{6n}}{3^{6n-1}}\right)=0.
$$
By comparing the coefficients of $u^{6n+3}$, we have 
$$
C_{6n+1}=-6C_{6n+4}+\sum_{k=1}^n C_{6(n-k)+4}G_{6k}(6k-1)(6k-2)\fr{1}{3^{6k-1}}.
$$
Therefore, 
$$
(6n+4)(2n+1)(6n+2)3^{6n}c_{6n+1}=-2\cdot 3^{6n}c_{6n+4}+\sum_{k=1}^n c_{6(n-k)+4}BH_{6k}(6k-1)(6k-2)3^{6(n-k)}\dbinom{6n+4}{6k}
$$
follows. This is the second equality of \eqref{rc}.

\eqref{sc} From Table~\ref{tc}, $c_{1}=1>0, c_{4}=-4<0$. We assume that $c_{1},c_{7}, \ldots, c_{6(n-1)+1}>0$ and $c_{4}, c_{10}, \ldots, c_{6(n-1)+4}<0$ for some $n\geq 1$. Then, since the first equality of \eqref{rc} and $BH_{6n}>0\ (n \geq 1)$ (by Lemma~\ref{sGD}), we have $c_{6n+1}>0$.
We also have $c_{6n+4}<0$, since the second equality of \eqref{rc}. By induction on $n$, \eqref{sc} is proved.

\eqref{se} 
By \eqref{sc}, for all $n\geq 0$, the Hurwitz coefficient for $-\Sl(-u)$ of $u^{3n+1}$ is $>0$. Since $\Cl'(-u)=-(-\Sl(-u))^2$ for all $n\geq 1$, the Hurwitz coefficient for $\Cl'(-u)=\sum_{n=1}^{\infty}e_{3n}(-1)^{3n-1}\fr{u^{3n-1}}{(3n-1)!}$ is $<0$.
Therefore, we have $(-1)^{3n}e_{3n}>0$. Since $\Cl(0)=1$, $e_{0}=1$.
Thus, $(-1)^{n}e_{3n}>0$ holds for $n\geq 0$.

\end{proof}

\begin{remark}
On the sign of $D_{6n+2}$ for $n\geq 0$, it is conjectured that $(-1)^nD_{6n+2}>0$, which is still open.
To prove this, we may have to relate it to some $L$-value to determine the sign of $D_{6n+2}$.
\end{remark}


\newpage
\subsection{Numerical tables of $c_{n}, d_{n}, G_{6n}$, and $BH_{6n}$}

\begin{table}[hbtp]
\caption{$c_{n}$}
\label{tc}
\centering
  $\begin{array}{lr}
    \hline
     n   &   c_{n}   \\
    \hline \hline
    1  &  1 \\
    4 &  -4  \\
    7 & 160 \\
   10 &  -20800  \\
   13 &  6476800 \\
16&  -3946624000  \\
19&  4161608704000 \\
22&  -6974121256960000  \\
25&  17455222222028800000 \\
28&  -62226770432344883200000  \\
31&  304379186781653598208000000 \\
34&  -1982049657077223312916480000000  \\
37&  16758824127564135479341219840000000 \\
40&  -180180787889254711099024290611200000000  \\
43&  2419729547280670262758159337861939200000000 \\
46&  -39971145354912684332749031990873817088000000000  \\
49&  801380022229927863218064428825418221486080000000000 \\
52&  -19272532158166604513119104829337619755759042560000000000  \\
55&  550209191558546672649809313730688966642337474150400000000000  \\
58&  -18474617726618802329201889210788021182113749879750656000000000000  \\
61&  723544590939960717069350289218516137325917930092538888192000000000000 \\
64& -32804563051713714135252946913038615626684567102838010344898560000000000000  \\
67& 
1710106255619904534930025572360427397429088633914126219002154844160000000000000\\
    \hline
  \end{array}$
\end{table}

\newpage

\begin{table}[htbp]
  \caption{$d_{n}$}
  \label{td}
  \centering
$
 \begin{array}{lr}
    \hline
    n  &  d_{n}  \\
    \hline \hline
    2  &  1/3\\
    &= 1/3\\
5& -10/21\\
    &= -2\cdot5/3\cdot7\\
8& -80/9\\
    &= -2^{4}\cdot5/3^{2}\\
11& 48400/273\\
     &= 2^{4}\cdot5^{2}\cdot11^{2}/3\cdot7\cdot13\\
14& 70400/3\\
     &= 2^{8}\cdot5^{2}\cdot11/3\\
17& -2309824000/1197\\
     &= -2^{9}\cdot5^{3}\cdot11\cdot17\cdot193/3^{2}\cdot7\cdot19\\
20& -2393600000/3\\
     &= -2^{12}\cdot5^{5}\cdot11\cdot17/3\\
23& 46641833216000/273\\
     &= 2^{11}\cdot5^{3}\cdot11^{2}\cdot17\cdot23\cdot3851/3\cdot7\cdot13\\
26& 4350492467200000/27\\
     &= 2^{16}\cdot5^{5}\cdot11^{2}\cdot17\cdot23\cdot449/3^{3}\\
29& -46343596783616000000/651\\
     &= -2^{17}\cdot5^{6}\cdot11^{2}\cdot17\cdot23\cdot29\cdot16493/3\cdot7\cdot31\\
32& -384236269846528000000/3\\
     &= -2^{20}\cdot5^{6}\cdot11^{2}\cdot17\cdot23\cdot29\cdot17093/3\\
35& 58356799095178616750080000000/575757\\
     &= 2^{20}\cdot5^{7}\cdot11^{3}\cdot17^{2}\cdot23\cdot29\cdot43\cdot1871\cdot34511/3^{2}\cdot7\cdot13\cdot19\cdot37\\
38& 930458061131469291520000000/3\\
     &= 2^{24}\cdot5^{7}\cdot11^{3}\cdot17^{2}\cdot23\cdot29\cdot701\cdot3947/3\\
41& -2526876563810119056660889600000000/6321\\
     &= -2^{25}\cdot5^{8}\cdot11^{3}\cdot17^{2}\cdot23\cdot29\cdot41\cdot431\cdot42521761/3\cdot7^{2}\cdot43\\
44& -17284095511545036218564608000000000/9\\
     &= -2^{28}\cdot5^{9}\cdot11^{4}\cdot17^{2}\cdot23\cdot29\cdot41\cdot331\cdot881\cdot977/3^{2}\\
47& 1027896133477097691463532176998400000000/273\\
     &= 2^{26}\cdot5^{8}\cdot11^{4}\cdot17^{2}\cdot23^{2}\cdot29\cdot41\cdot47\cdot313\cdot1001523179/3\cdot7\cdot13\\
50& 80306491124224476202875370864640000000000/3\\
     &= 2^{32}\cdot5^{10}\cdot11^{5}\cdot17^{2}\cdot23^{2}\cdot29\cdot41\cdot47\cdot61\cdot2029\cdot11243/3\\
53& -272275412126083890645192489250796339200000000000/3591\\
     &= -2^{33}\cdot5^{11}\cdot11^{4}\cdot17^{3}\cdot23^{2}\cdot29\cdot41\cdot47\cdot53\cdot1201\cdot4795973261/3^{3}\cdot7\cdot19\\
56& -2287147842776869352629978007993738854400000000000/3\\
     &= -2^{36}\cdot5^{11}\cdot11^{5}\cdot17^{3}\cdot23^{2}\cdot29\cdot41\cdot47\cdot53\cdot137\cdot4013306567/3\\
    \hline
 \end{array}
$
\end{table}

\newpage

\begin{table}[hbtp]
  \caption{$G_{6n}$, and $BH_{6n}$}
  \label{tg}
  \centering
$
\begin{array}{lrr}
    \hline
    n  &  G_{n} & BH_{n}  \\
    \hline \hline
6 & 27/140 &972/7  \\
12& 729/112112 & 283435200/91\\
18 &19683/184352896 & 90914674752000/133\\
24& 9034497/2269709478400 & 224740203206066380800/91\\
30& 961376769/9936195998310400 & 5569206003076317032217600000/217\\
36& 45251875376667/15343234021377167360000 & 491303185206504946068649519911075840000/447811\\
42& 3409353379806993/43280194527500713689088000 & 33314057040643263742571290653335322624000000/301\\

\hline
 \end{array}
$
\end{table}




\end{document}